\documentclass[12pt]{article}
\usepackage[dvips]{graphicx}
\usepackage{amsmath, amssymb}
\usepackage{amsthm}
\usepackage{bm}

\newtheorem{theorem}{Theorem}[section]
\newtheorem{lemma}{Lemma}[section]
\newtheorem{corollary}{Corollary}[section]

\newtheorem{proposition}{Proposition}[section]

\newcommand{\cK}{{\cal K}}
\newcommand{\cP}{{\cal P}}

\newcommand{\tp}{^{\mathrm t}}
\newcommand{\tpl}{\;{\mathrm t}}
\newcommand{\tr}{\mathop{{\rm tr}}}
\newcommand{\co}{\mathop{{\rm co}}}

\newcommand{\lmaxN}{\lambda_{\max,N}}
\newcommand{\lminN}{\lambda_{\min,N}}

\newlength{\IndentI}
\newlength{\IndentII}
\newlength{\IndentIII}
\newlength{\WidthI}
\newlength{\WidthII}
\newlength{\WidthIII}
\title{Sequential optimizing strategy in  multi-dimensional bounded forecasting games}

\author{
Masayuki Kumon\footnote{Risk Analysis Research Center, The Institute of Statistical Mathematics, Research Organization of Information and Systems,  masayuki\_kumon@smile.odn.ne.jp}  
\\
Akimichi Takemura
\footnote{
Graduate School of Information Science and Technology,
University of Tokyo, \   
7-3-1 Hongo, Bunkyo-ku, Tokyo 113-8656, JAPAN, \ 
takemura@stat.t.u-tokyo.ac.jp}
\\
Kei Takeuchi
\footnote{
Emeritus, Graduate School of Economics, 
University of Tokyo}
}

\date{November 2009}

\begin{document}
\maketitle

\begin{abstract}
  We propose a sequential optimizing betting strategy in the
  multi-dimensional bounded forecasting game in the framework of
  game-theoretic probability of Shafer and Vovk (2001). By studying
  the asymptotic behavior of its capital process, we prove a generalization
  of the strong law of large numbers, where the convergence rate of the sample
  mean vector depends on the growth rate of
  the quadratic variation process.  The
  growth rate of the quadratic variation process may be slower than
  the number of rounds or may even be zero.
  We also introduce an 
  information criterion for selecting efficient betting
  items. These results are then applied to multiple asset trading strategies
  in discrete-time and 
  continuous-time games. In the case of continuous-time game 
  we present a measure of the jaggedness of a
  vector-valued continuous process.
  Our results are examined by several numerical examples.
\end{abstract}


\section{Introduction}
\label{sec:intro}

Since the advent of the game-theoretic probability and finance 
by Shafer and Vovk \cite{shafer/vovk:2001}, the field has been expanding
rapidly. The present authors have been contributing to this emerging field
by  showing the essential role of the Kullback-Leibler divergence
for the strong law of large numbers (SLLN) \cite{kumon-takemura,ktt:2007b}
and by proposing a new approach
to continuous-time games \cite{takeuchi-kumon-takemura-bernoulli,takeuchi-kumon-takemura-markov}.
Our approach to continuous-time games has been  further developed by
V.~Vovk \cite{vovk07-stochastics, vovk08-ecp,vovk0904.4364}.

In this paper we propose a sequential optimizing betting strategy for
the multi-dimensional bounded forecasting game in discrete time and
apply it as a high-frequency limit order type betting strategy for
vector-valued continuous price processes.  

Our strategy is very flexible and the analysis of its asymptotic
behavior allows us to generalize game-theoretic statements of SLLN to
a wide variety of cases.  SLLN for the bounded forecasting game is
already established in Chapter 3 of \cite{shafer/vovk:2001}. In
\cite{kumon-takemura} we gave a simple strategy forcing SLLN
with the rate of $O(\sqrt{\log n/n})$, where $n$ is the number of
rounds.  However the convergence rate of SLLN should depend on the
growth rate of the quadratic variation process.  For example, in view of
Kolmogorov's three series theorem (e.g.\ Section IV.2 of
\cite{shiryaev:1996}), the sum $s_n = x_1 + \dots + x_n \in {\mathbb R}^1$ of 
centered  independent 
measure-theoretic random variables converges a.s.\ 
if the sum of their variances converges (i.e.\ $\sum_n{\rm
  Var}(x_n) < \infty$).  Therefore in this case the sample average
$\bar x_n=s_n/n$ is of order $O(1/n)$.  By our sequential optimizing
betting strategy, we can give a unified game-theoretic treatment on
the asymptotic behavior of $s_n$, which depends on the asymptotic
behavior of $\sum_{i=1}^n x_i^2$ as $n \rightarrow\infty$.  

The strength of our results can be seen when we interpret our results
in the standard measure-theoretic framework.  Let
$s_n = x_1 + \dots + x_n$ be a 
one-dimensional measure-theoretic  martingale w.r.t.\ 
a filtration $\{{\cal F}_n\}$ 
with uniformly  bounded differences $\Vert x_n\Vert \le 1$, a.e. 
Let $V_n=x_1^2 + \dots + x_n^2$.  Then with probability one the sequence 
$\Vert s_n\Vert/\sqrt{\max(0,V_n \log V_n)}$, $n=1,2,\dots$, is bounded.
See Proposition \ref{prop:measure} below.  Note that in this statement no assumption is
made on the growth rate of $V_n$. The rate itself may be random.

{}From more practical viewpoint, our sequential betting strategy is very
simple to implement even for high dimensions and shows a very
competitive performance when applied to various price processes.  In
Section \ref{sec:numerical} we compare the performance of our strategy
with the well-known universal portfolio strategy developed by Thomas
Cover and collaborators (\cite{cover:1991,cover-ordentlich,ordentlich-cover,cover-thomas:2006}).  
The performance
of our sequential betting strategy is competitive against the universal
portfolio.  Note that the numerical integration needed for 
implementing universal portfolio is computationally heavy for high dimensions.

When we can bet on a large number of price processes, it is not always
best to form a portfolio comprising all price processes, because
estimating the best weight vector for the price processes might take a
long time.  By approximating the capital process  of our sequential
optimizing strategy, we will introduce an information criterion 
for selecting price processes in a portfolio.

The organization of this paper is as follows.  In Section
\ref{sec:sequential-strategy} we formulate the multi-dimensional
bounded forecasting game,
introduce our sequential optimizing strategy 
and state our main theoretical result.
In Section \ref{sec:structure}
we give a proof of our result 
by analyzing asymptotic behavior of its capital process.
We also introduce
an information criterion for selecting
efficient betting items.  These results are then applied to multiple asset
trading games in 
Section \ref{sec:high}.
In Section \ref{sec:high} 
we formulate the multiple asset trading game in continuous time and 
based on high-frequency limit order type betting strategies 
we present a measure for the jaggedness of a path of
a vector-valued continuous process. 
In Section \ref{sec:generality}, as indicating the generality of our
results, we provide a multiple type of Girsanov's theorem for geometric Brownian motion and an
argument concerning the mutual information quantity between betting
games.  In Section \ref{sec:numerical} we give numerical results for
several Japanese stock price processes.  We conclude the paper with
some remarks in Section \ref{sec:discussion}.

\section{A sequential optimizing strategy and its implication to strong law of large numbers}
\label{sec:sequential-strategy}

We treat the following type of discrete time bounded forecasting game
between Skeptic and Reality. 
$\cK_0$ is the initial capital of Skeptic, 
$D$ is a compact region 
in ${\mathbb R}^d$ such that its convex hull 
$\co D$ 
contains the origin in its interior,  and $\cdot$ denotes the standard
inner product of ${\mathbb R}^d$.

\bigskip\noindent
\textsc{Discrete Time Bounded Forecasting Game} \\
\textbf{Protocol:}

\parshape=6
\IndentI   \WidthI
\IndentI   \WidthI
\IndentII  \WidthII
\IndentII  \WidthII
\IndentII  \WidthII
\IndentI   \WidthI
\noindent
${\cal K}_0 :=1$.\\
FOR  $n=1, 2, \dots$:\\
  Skeptic announces $M_n\in{\mathbb R}^d$.\\
  Reality announces $x_n\in D$.\\
  ${\cal K}_n = {\cal K}_{n-1}+ M_n\cdot x_n$.\\
END FOR\\

In this paper we regard $d$-dimensional
vectors such as $x_n = (x_n^1, \dots, x_n^d)\tp$ as column vectors
with $\tp$ denoting the transpose. $\Vert x\Vert = \sqrt{x\tp x}=\sqrt{x\cdot x}$ denotes
the Euclidean
norm of $x$.  
Letting $\alpha_n = M_n /\cK_{n-1}$, we can rewrite
Skeptic's capital as ${\cal K}_n = {\cal K}_{n-1}(1 + \alpha_n\cdot
x_n),\ \alpha_n \in {\mathbb R}^d$.  
In the protocol, we require that Skeptic observes his collateral duty,
in the sense that $\cK_n \ge 0$ for all $n$ irrespective of 
Reality's moves $x_1, x_2, \dots$.

For constructing a strategy of Skeptic, consider that
Skeptic himself generates
`training data' $\{x_{-n_0+1},x_{-n_0+2}, \dots, x_0\}$ of size $n_0 \ge d+1$.
This operation is similar to a construction of a prior distribution in Bayesian statistics, where
a prior distribution can be specified by a set of prior observations.
Throughout this paper we fix an arbitrarily $\epsilon_0\in (0,1)$ and 
choose the training data  $\{x_{-n_0+1}, \dots, x_{0}\}$ in such a way that
\begin{equation}
\label{eq:2-0}
1 + \alpha\cdot x_n \ge 0, \ n = -n_0+1,  \dots, 0, \quad \Rightarrow\quad
1+ \alpha\cdot x \ge \epsilon_0, \   \forall x\in D.
\end{equation}
Let $P^d_{n_0,\epsilon_0}=
\{ \alpha \mid 1 + \alpha\cdot x_n \ge 0, \ n = -n_0+1, \dots, 0\}$.
Then (\ref{eq:2-0}) is equivalent to 
\[
P^d_{n_0,\epsilon_0} \subset -(1-\epsilon_0)(\co  D)^\perp,
\]
where $(\co  D)^\perp$ denotes the convex dual of $\co D$.
For example, for $d=1$ and $D=[-1,1]$, we can take
$x_{-1} = 1/(1-\epsilon_0)$ and $x_0 = -1/(1-\epsilon_0)$.
Then $\alpha$ has to satisfy $|\alpha| \le  1-\epsilon_0$ and
the right-hand side of (\ref{eq:2-0}) holds.  For general $D \subset {\mathbb R}^d$ 
let $\bar\delta=\max_{x\in D} \Vert x \Vert$ and let $c=\bar\delta\sqrt{d}/(1-\epsilon_0)$.
Then we can take
$n_0 = 2d$ training vectors as 
\[
(0,\dots,0,\pm c,0,\dots,0)\tp,
\]
where  $c$ is in the $i$-th coordinate ($1\le i\le d$).
Then each element $\alpha^i$, 
$1\le i\le d$, of 
$\alpha=(\alpha^1, \dots, \alpha^d)\tp$ has to satisfy
$|\alpha^i|\le 1/c$
and 
$\Vert\alpha\Vert \le (1-\epsilon_0)/\bar\delta$. Hence the right-hand side  of (\ref{eq:2-0}) holds
by Cauchy-Schwarz inequality. It should also be noted that (\ref{eq:2-0}) implies that
the training vectors span the whole ${\mathbb R}^d$.

The strategy with a constant vector
$\alpha_n \equiv \alpha
\in {\mathbb R}^d$ 
is called a constant proportional betting strategy.  For $N\ge 0$ we  define
\begin{equation}
\label{eq:phiN}
\Phi_{0,N}(\alpha) = 
\sum_{n=-n_0+1}^N \log (1 + \alpha\cdot x_n), 
\end{equation}
which is the log capital at round $N$ under 
the constant proportional betting strategy, including the training data. 
We add `$0$' to the subscript to indicate that the training data are included in a
summation. 
Since the game starts at time $1$, actually
the log capital of the constant proportional betting 
strategy is $\Phi_{0,N}(\alpha)-\Phi_{0,0}(\alpha)=
\sum_{n=1}^N \log (1 + \alpha\cdot x_n)$.

Let us consider the maximization of $\Phi_{0,N}(\alpha)$ with respect to $\alpha \in {\mathbb R}^d$.  
The maximum corresponds to the log capital at time $N$ 
of a `hindsight' constant proportional betting strategy.
Note that $\Phi_{0,N}(\alpha)$ is a strictly concave function
of $\alpha$. 
The condition (\ref{eq:2-0})  ensures that
the maximum of $\Phi_{0,N}(\alpha)$ is attained  at
the unique point $\alpha = \alpha_N^*$
in the interior of $P^d_{n_0,\epsilon_0}$ 
so that
\begin{align}
\label{eq:2-1}
&
\frac{\partial \Phi_{0,N}}{\partial \alpha}\Bigm|_{\alpha = 
\alpha_N^*} 
= 
\sum_{n=-n_0+1}^N \frac{x_n}
{1 + \alpha_N^*\cdot x_n} = 0.
\end{align}
{}From numerical viewpoint we note that the numerical maximization of 
$\Phi_{0,N}(\alpha)$ is straightforward even in high dimensions.

We now define {\it sequential optimizing strategy} (SOS)
of Skeptic, which is a realizable strategy unlike 
the hindsight strategy.  It is given by
\begin{align}
\label{eq:2-3b}
\alpha_n = \alpha_{n-1}^* ,\ \ n \ge 1.
\end{align}
The idea of SOS is very simple.  We employ the empirically best constant proportion
until the previous round for betting at the current round.
Note that SOS depends on the choice of the training data.
Skeptic's log capital $\log {\cal K}_{1,N}^*$  at round $N$
under SOS 
is written as
\[
\log \cK_{1,N}^*=\sum_{n=1}^N \log (1 + \alpha_{n-1}^*\cdot x_n).
\]

Let $\xi=x_1 x_2 \dots \in D^\infty$ 
denote a {\em path}, which is an infinite sequence of Reality's moves.
The set $\Omega=D^\infty$ of paths 
is called the sample space and any subset $E$ of $\Omega$ is called
an event.
$\xi^n=x_1\dots x_n$ denotes a partial path of Reality until the round $n$.
A strategy $\cP$ specifies $\alpha_n$ in terms of $\xi^{n-1}$, i.e.\ $\alpha_n=\cP(\xi^{n-1})$.
The capital process under $\cP$ is given as 
$\cK^{\cP}_{1,N}=\prod_{n=1}^N (1+\cP(\xi^{n-1})\cdot x_n)$.
$\cP$ is called prudent, if Skeptic observes his collateral duty
by $\cP$, i.e.\  $\cK^{\cP}_{1,N} \ge 0$, $\forall N\ge 0$,   irrespective of
Reality's moves $x_1, x_2, \dots$.  In this paper we only consider prudent strategies
of Skeptic.
We say that Skeptic can weakly force an event $E\subset \Omega$ by a strategy $\cP$ if
$\limsup_N \cK^{\cP}_{1,N}=\infty$ for every $\xi\not\in E$.
As in Section \ref{sec:intro} we write
\begin{equation}
\label{eq:snvn}
s_N = x_1 +\dots+ x_N \in {\mathbb R}^d , 
\quad V_N = x_1 x_1\tp + \cdots + x_N x_N\tp \quad(: d\times d) .
\end{equation}
Then $\tr V_N = \sum_{n=1}^N \Vert x_n\Vert^2$.
We are now ready to state our main theorem.

\begin{theorem}
\label{thm:gen-slln}
By the sequential optimizing strategy Skeptic can
weakly force 
\begin{equation}
\label{eq:gen-slln}
E: \limsup_N \frac{\Vert s_N\Vert}{\sqrt{\max(1,\tr V_N \log(\tr V_N)})} < \infty.
\end{equation}
\end{theorem}

The maximum in the denominator is needed only for the case that
$\sup_N \tr V_N \le 1$, such as $0\in D$ and 
Reality always chooses  $x_n\equiv 0$.
It is important to emphasize that $E$ in (\ref{eq:gen-slln}) is weakly forced
irrespective of the rate of growth of $\tr V_N$, including the zero-growth case, i.e.\ 
the case that $\tr V_N$ converges to a finite value.
A measure-theoretic interpretation of our result shows the flexibility of
our result. When $x_n$'s are measure-theoretic martingale differences, then the capital process
under SOS is a non-negative measure-theoretic martingale, which converges to a finite
value almost surely.  Therefore as in Chapter 8 of \cite{shafer/vovk:2001} we have 
the following proposition.  We use the same notation as above.

\begin{proposition}
\label{prop:measure}
Let $s_n = x_1 + \dots + x_n$ be a 
$d$-dimensional measure-theoretic  martingale w.r.t.\ 
a filtration $\{{\cal F}_n\}$. 
Assume that 
the differences $x_n\in D$ are uniformly bounded a.e.. 
Then with probability one the sequence 
$\Vert s_n\Vert /\sqrt{\max(1,\tr V_n \log (\tr V_n))}$, $n=1,2,\dots$, is bounded.
\end{proposition}

Let $\lmaxN$ and $\lminN$  denote the maximum and the minimum eigenvalues of $V_N$.
Consider the event
\begin{equation}
\label{eq:min-log-max}
E' : \lim_N \frac{\log \lmaxN}{\lminN} = 0.
\end{equation}
Theorem \ref{thm:gen-slln} gives only the order of $s_N$. 
If we condition the paths on
the event $E'$, then we can derive a more accurate numerical bound as follows.

\begin{theorem}
\label{thm:gen-slln2}
By the sequential optimizing strategy Skeptic can
weakly force
\[
E'  \ \Rightarrow \ \limsup_N \frac{s_N\tp V_N^{-1} s_N}{\log |V_N|} \le 1.
\]
\end{theorem}

This theorem follows from the fact that
on  $E'$ Skeptic can weakly force $\alpha_N^* \rightarrow 0$, as shown in the
proof of this theorem in Section 
\ref{subsec:approximation}.

Note that $\lminN\rightarrow\infty$ on $E'$.
Note also that $E$ in (\ref{eq:gen-slln}) holds if and only if 
$\limsup_N \Vert s_N\Vert/\allowbreak\sqrt{\max(1,\lmaxN \log\lmaxN)} < \infty$, 
because $\lmaxN \le \tr V_N \le d \lmaxN$. 
Hence on $E'$ 
we have
\begin{align*}
1 &\ge \limsup_N \frac{s_N\tp V_N^{-1} s_N}{\log |V_N|}  \ge \limsup_N
\frac{\Vert s_N\Vert^2}{d \lmaxN \log \lmaxN}.
\end{align*}
Therefore, although we only have a conditional statement in
Theorem \ref{thm:gen-slln2}, it gives a more accurate numerical bound 
than Theorem 
\ref{thm:gen-slln}.

\section{Proof of the theorem and some other results on sequential optimizing strategy}
\label{sec:structure}

In this section we provide proofs of the above theorems  and present
other results on the sequential optimizing strategy. 
For readability, we divide the section into several subsections.

\subsection{Properties of $\alpha_N^*$ and the empirical risk neutral distribution}
Let $\delta_x$ denote a unit point mass at $x\in {\mathbb R}^d$ and let
$g_N = \sum_{n=-n_0+1}^N \delta_{x_n}/(N+n_0)$ denote the empirical distribution of the training data and Reality's moves 
$x_1,\dots,x_N$ up to round $N$. 
In view of (\ref{eq:2-1}) we define the {\em empirical risk neutral distribution}
$g_N^*$ up to round $N$ by
\[
g_N^* = \frac{1}{N+n_0}\sum_{n=-n_0+1}^N \frac{\delta_{x_n}}
{1 + \alpha_N^*\cdot x_n}.
\]
For notational simplicity we omit `0' from the subscript of $g_N$ and $g_N^*$, although
they involve the training data.
$g_N^*$ is indeed a probability measure, because by (\ref{eq:2-1}) we have
\[
\sum_{x_n} g_N(\{x_n\}) = \frac{1}{N+n_0} \sum_{n=-n_0+1}^N \frac{1}{1+\alpha_N^* \cdot x_n}
=\frac{1}{N+n_0} \sum_{n=-n_0+1}^N 
\frac{1+\alpha_N^* \cdot x_n}{1+\alpha_N^* \cdot x_n}=1,
\]
where the summation on the left-hand side is over distinct values of $x_n$, 
$n=-n_0+1,\dots, N$. 
By $E_{g_N^*}[\cdot]$ we denote the expected value under $g_N^*$.  Then
(\ref{eq:2-1}) is written as  $E_{g_N^*}[x]=0$.

The log capital $\log {\bar {\cal K}}_{0,N}^*=\Phi_{0,N}(\alpha_N^*)$ of the 
constant hindsight strategy $\alpha_N^*$ up to round $N$ including the training data
is expressed as
\begin{align}
\label{eq:2-3}
\log {\bar {\cal K}}_{0,N}^* = \Phi_{0,N}(\alpha_N^*) = 
(N+n_0)\sum_{x_n} g_N(\{x_n\})\log \frac{g_N(\{x_n\})}{g_N^*(\{x_n\})} = 
(N+n_0) D(g_N \| g_N^*),
\end{align}
where 
$D(g_N \| g_N^*)$ denotes 
the Kullback-Leibler divergence between two probability distributions
$g_N$ and $g_N^*$.

Now note that $\log (1+\alpha_{n-1}^* \cdot x_n)  = \Phi_n(\alpha_{n-1}^*)-\Phi_{n-1}(\alpha_{n-1}^*)$.
By summation by parts, the difference 
$\log {\bar {\cal K}}_{0,N}^* - \log {\cal K}_{1,N}^*$ between
the hindsight strategy and SOS can be expressed as
\begin{align}
\label{eq:2-5}
\log {\bar {\cal K}}_{0,N}^* - \log {\cal K}_{1,N}^* = 
\sum_{n=1}^N \Delta \Phi_n + \Phi_{0,0}(\alpha_0^*),
\end{align}
where $\Delta \Phi_n = \Phi_n(\alpha_n^*) - \Phi_n(\alpha_{n-1}^*) \ge 0$
and $\Phi_{0,0}(\alpha_0^*)$ is a constant depending only on the
training data.
We will analyze the behavior the log capital $\log {\cal K}_{1,N}^*$  of SOS 
by analyzing
$\log {\bar {\cal K}}_{0,N}^*$ and 
$\sum_{n=0}^N \Delta \Phi_n$.  

We call
\[
{\bar V}_N^* = \frac{1}{N+n_0} V_{0,N}^*= E_{g_N^*}[x x\tp] 
= \frac{1}{N+n_0} \sum_{n=-n_0+1}^N \frac{x_n x_n\tp}{1+\alpha_N^* \cdot x_n}
\]
the empirical risk neutral covariance matrix for Reality's moves up to round $N$.
Write 
\[
s_{0,N}=\sum_{n=-n_0+1}^N x_n, \quad \bar x_{0,N}=\frac{1}{N+n_0} s_{0,N}.
\]
Noting 
$g_N(\{x_n\}) = (1 + \alpha_N^*\cdot x_n)g_N^*(\{x_n\})$, we have
\begin{align*}
\bar x_{0,N} = E_{g_N}[x] = E_{g_N^*}[(1 + \alpha_N^*\cdot x)x] =
E_{g_N^*}[x] +  E_{g_N^*}[xx\tp ] \alpha_N^* = E_{g_N^*}[xx\tp ] \alpha_N^*.
\end{align*}
Therefore $\alpha_N^*$ is expressed as
\begin{align}
\label{eq:2-8b}
\alpha_N^* 
= {\bar V}_N^{*-1}{\bar x}_{0,N} = 
V_{0,N}^{*-1} s_{0,N}.
\end{align}
Since ${\bar V}_N^*$ itself contains $\alpha_N^*$, (\ref{eq:2-8b})
does not give an explicit expression of $\alpha_N^*$.  However it is 
a very useful exact relation for our analysis.

We now consider $\Delta \alpha_n^* = \alpha_n^* - \alpha_{n-1}^*$.  In the following we use
the notation
\[
x_n(\alpha)= \frac{x_n}{1+\alpha\cdot x_n}.
\]
Taking the difference of the following two equalities 
\[ 
0 = \sum_{i=-n_0+1}^n x_i(\alpha_n^*), \qquad 
0 = \sum_{i=-n_0+1}^{n-1} x_i(\alpha_{n-1}^*)
\]
we obtain
\[
0= \sum_{i=-n_0+1}^{n-1} x_i \frac{(\alpha_{n-1}^* - \alpha_n^*)\cdot x_i}
      {(1+\alpha_{n-1}^*\cdot x_i)(1+\alpha_n^*\cdot x_i)}
 + x_n(\alpha_n^*).
\]
Therefore
\[
\Big(\sum_{i=-n_0+1}^{n-1}  \frac{x_i x_i\tp}
     {(1+\alpha_{n-1}^*\cdot x_i)(1+\alpha_n^*\cdot x_i)}\Big)
(\alpha_n^* - \alpha_{n-1}^*) 
= x_n(\alpha_n^*).
\]
Note that the denominator on the left-hand side  is a scalar and the
matrix on the left-hand side  is positive definite.
Then 
\begin{equation}
\label{eq:delta-alpha}
\Delta\alpha_n^* = V_{0,n-1}(\alpha_{n-1}^*, \alpha_n^*)^{-1} x_n(\alpha_n^*),
\end{equation}
where 
$
V_{0,n-1}(\alpha,\beta)=\sum_{i=-n_0+1}^{n-1}  x_i(\alpha) x_i(\beta)\tp
$.

Concerning the behavior of $\Delta\alpha_n^*$ we state the following 
lemma, which will be used in Section \ref{subsec:approximation}.


\begin{lemma}
\label{lem:convergence2}
$\lim_n \Delta\alpha_n^* =0$
for every $\xi\in D^\infty$.
\end{lemma}

We give a proof of this lemma in Appendix.


\subsection{Bounding the difference between the hindsight strategy and SOS from above}
We now give a detailed analysis of 
$\sum_{n=1}^N \Delta \Phi_n$ on the right-hand side  of
(\ref{eq:2-5}) and bound it from above.
We note the following simple fact on $\Delta \Phi_n$:
\begin{align}
\Delta \Phi_n&=\Phi_n(\alpha_n^*)-\Phi_n(\alpha_{n-1}^*) = 
\Phi_{n-1}(\alpha_n^*)-\Phi_{n-1}(\alpha_{n-1}^*)  
+ \log \frac{1+\alpha_n^*\cdot x_n}{1+\alpha_{n-1}^*\cdot x_n}\nonumber \\
&\le 
\log \frac{1+\alpha_n^*\cdot x_n}{1+\alpha_{n-1}^*\cdot x_n}
= \log \big( 1 + \frac{\Delta\alpha_n^* \cdot x_n}{1+\alpha_{n-1}^*\cdot x_n}\big),
\label{eq:delta-phi-upper-0}
\end{align}
where the inequality holds 
since $\alpha_{n-1}^*$ maximizes $\Phi_{n-1}(\alpha)$.  Substituting 
(\ref{eq:delta-alpha}) into the right-hand side we obtain
\begin{equation}
\Delta \Phi_n \le 
 \log \big( 1 + x_n(\alpha_n^*)\tp 
V_{0,n-1}(\alpha_{n-1}^*, \alpha_n^*)^{-1} x_n(\alpha_{n-1}^*)\big).
\label{eq:delta-phi-upper-1}
\end{equation}
Note that  we can also rewrite
\begin{align}
1 + x_n(\alpha_n^*)\tp 
V_{0,n-1}(\alpha_{n-1}^*, \alpha_n^*)^{-1} x_n(\alpha_{n-1}^*)
&= \frac{|V_{0,n}(\alpha_{n-1}^*, \alpha_n^*)|}{|V_{0,n-1}(\alpha_{n-1}^*, \alpha_n^*)|},
\label{eq:det-form2}
\end{align}
where 
we used 
a well-known relation between determinants (e.g.\ Corollary A.3.1 of \cite{anderson-3rd}).
%


Let 
\begin{equation}
\label{eq:C1}
C_1 = \max\big(\sup_{\alpha\in -(\co D)^\perp,  \; x\in D} (1+\alpha\cdot x),
\sup_{-n_0+1 \le n\le 0\atop\alpha\in P_{n_0,\epsilon_0}^d} (1+\alpha\cdot x_n)
\big),
\end{equation}
which is a constant depending only on the training data.
The first argument $C_{1,0}=\sup_{\alpha\in -(\co D)^\perp,  \; x\in D} (1+\alpha\cdot x)$
on the right-hand side of (\ref{eq:C1})
corresponds to the maximum one-step growth rate of Skeptic's capital under the
collateral duty and
$C_{1,0}$ equals 2 if $D$ is symmetric w.r.t.\ the origin.
$C_{1,0}$ may be large if  $D$ is highly asymmetric w.r.t.\ the origin.
For example, for $d=1$ and $D=[-0.1,1]$, 
we have $C_{1,0}=11$.

For two symmetric matrices $A,B$, let $A\ge B$ mean that
$A-B$ is non-negative definite.   Then
\[
V_{0,n-1}(\alpha_{n-1}^*, \alpha_n^*) \ge  \frac{1}{C_1^2} V_{0,n-1}, 
\]
where $V_{0,n-1}=V_{0,n-1}(0,0)=\sum_{i=-n_0+1}^n x_i x_i\tp$.
Note that $V_{0,n-1}$ is positive definite because of the training data, 
although $V_{n-1}$ in (\ref{eq:snvn}) may be singular.
Note also that 
$
1 + \alpha_m^*\cdot x_n \ge \epsilon_0
$ for $m,n\ge 1$.  Therefore
\[
x_n(\alpha_n^*)\tp 
V_{0,n-1}(\alpha_{n-1}^*, \alpha_n^*)^{-1} x_n(\alpha_{n-1}^*)
\le C_2  x_n\tp  V_{0,n-1}^{-1} x_n, \qquad C_2 = \frac{C_1^2}{\epsilon_0^2}.
\]
%
Hence we can bound 
\[
\sum_{n=1}^N \Delta \Phi_n \le 
\sum_{n=1}^N 
\log (1 + C_2 x_n\tp  V_{0,n-1}^{-1} x_n).
\]
Write $a_n=x_n\tp V_{0,n-1}^{-1} x_n \ge 0$. 
Note that $1+a_n=|V_{0,n}|/|V_{0,n-1}|$.
Also for $c\ge 1$ and $a\ge 0$ we have
$1+ac \le (1+a)^c$ and hence
\[
\log (1+ac) \le c \log (1+a).
\]
Therefore
\[
\sum_{n=1}^N \log (1+C_2 a_n) \le C_2 \sum_{n=1}^N \log (1+a_n)
= C_2\log  \frac{|V_{0,N}|}{|V_{0,0}|}.
\]
Now we have proved the following lemma.

\begin{lemma} 
\label{lem:diff1}
The difference
$\sum_{n=1}^N \Delta \Phi_n$ on the right-hand side  of
(\ref{eq:2-5}) is bounded from above as
\[
\sum_{n=1}^N \Delta \Phi_n \le  
C_2( \log |V_{0,N}| - \log |V_{0,0}|).
\]
\end{lemma}

Since $|V_{0,N}|$ involves the training data, 
for simplicity in our statement  we further bound 
it as follows.
By the inequality between the geometric mean and arithmetic mean we have
$|V_{0,N}|^{1/d} \le  \tr V_{0,N}/d$. Hence
\begin{align*}
\log |V_{0,N}| 
&\le d \log  \tr V_{0,N} - d \log d
= d \log (\tr V_N + \tr V_{0,0}) -  d \log d\\
& \le 
d \log (\tr V_N + \tr V_{0,0}).
\end{align*}
If $\tr V_N \le 1$, then 
$\log (\tr V_N + \tr V_{0,0}) \le \log(1+\tr V_{0,0}) \le \tr V_{0,0}$.
On the other hand if  $\tr V_N >  1$, then
\begin{align*}
\log (\tr V_N + \tr V_{0,0}) &=\log \tr V_N + \log (1+\frac{\tr V_{0,0}}{\tr V_N})
\le \log \tr V_N + \frac{\tr V_{0,0}}{\tr V_N} \\
&\le
\log \tr V_N + \tr V_{0,0}.
\end{align*}
Therefore for both cases
$\log |V_{0,N}| \le d \max(0,\log\tr V_N) + d \tr V_{0,0}$.
Let $C_3=C_2(d \tr V_{0,0} - \log |V_{0,0}|)$.
In summary we have the following bound.  
\begin{lemma} 
\label{lem:diff2}
The difference
$\sum_{n=1}^N \Delta \Phi_n$ on the right-hand side  of
(\ref{eq:2-5}) is bounded from above as
\begin{equation}
\label{eq:boundphin}
\sum_{n=1}^N \Delta \Phi_n \le  
d C_2 \max(0,\log\tr V_N) + C_3,
\end{equation}
where $C_2, C_3$ depend only on the training data.
\end{lemma}


Note that (\ref{eq:boundphin}) is true even for the case that $\lim_N \tr V_N < \infty$.
Note also that $\log \tr V_N$ is  of order $O(\log N)$
even for $V_N=O(N^{\gamma})$, $0< \gamma < 1$.
However for $\tr V_N=\log N$ we have $\log \tr V_N=\log \log N$.





\subsection{Bounding the hindsight strategy from below}

In this subsection we bound the hindsight strategy from below and thus finish
the proof of Theorem \ref{thm:gen-slln}.  

Consider the function $(1+t)\log(1+t)$, $t > -1$. By Taylor expansion we have
\[
(1+t)\log (1+t)= t + \frac{1}{2} \frac{t^2}{1+\theta t},   \qquad 0 < \theta <1.
\]
By the definition of $C_1$ in (\ref{eq:C1}) we have
\[
(1+\alpha_N^* \cdot x_n )\log (1+\alpha_N^* \cdot x_n)
\ge \alpha_N^* \cdot x_n + \frac{(\alpha_N^* \cdot x_n)^2}{2C_1} ,
\quad \forall N\ge 1, \ -n_0 +1 \le \forall n \le N,
\]
and
\begin{align}
\Phi_{0,N}(\alpha_N^*)&= \sum_{n=-n_0+1}^N (1+\alpha_N^* \cdot x_n) \log(1+\alpha_N^* \cdot x_n) 
\frac{1}{1+\alpha_N^* \cdot x_n} \nonumber \\
&\ge  \sum_{n=-n_0+1}^N \frac{\alpha_N^* \cdot x_n}{1+\alpha_N^* \cdot x_n}
+ \frac{1}{2C_1}\sum_{n=-n_0+1}^N \frac{(\alpha_N^* \cdot x_n)^2}{1+\alpha_N^* \cdot x_n}
\nonumber \\
&= \frac{1}{2C_1}\sum_{n=-n_0+1}^N \frac{(\alpha_N^* \cdot x_n)^2}{1+\alpha_N^* \cdot x_n},
\label{eq:bound1}
\end{align}
where we have used the fact $E_{g_N^*}[x]=0$.
By (\ref{eq:2-8b})
the summation on the right-hand side can be written as
\begin{equation}
\label{eq:empirical-bound-10-2}
\sum_{n=-n_0+1}^N \frac{(\alpha_N^* \cdot x_n)^2}{1+ \alpha_N^* \cdot x_n}
= \alpha_N^{*\tpl} V_{0,N}^* \alpha_N^* = \alpha_N^{*\tpl} s_{0,N} 
= s_{0,N}\tp  V_{0,N}^{*-1}  s_{0,N}.
\end{equation}

In analyzing the behavior of (\ref{eq:empirical-bound-10-2})
we need to be careful about the following fact: 
$1+\alpha_N^*\cdot x_n$, $n\le 0$, may be arbitrarily close to zero
for the training data $x_{-n_0+1}, \dots, x_0$.
%
In particular
we might have different behavior between eigenvalues of
$V_{0,N}^*$ and and those of  $V_N$.
To assess the effect of training data  let
\[
A_N = \sum_{n=-n_0+1}^{0} \frac{1}{1+\alpha_N^* \cdot x_n}
\]
and define the following event 
\[
E_1 : \limsup_N \frac{A_N}{\max(0,\log \tr V_N)} <\infty .
\]
Again $\max$ is needed only for the case that $\tr V_N \le 1$ for all $N$.
We now show that Skeptic can weakly force $E_1$.
Fix an arbitrary $\xi\in E_1^c$, where $E_1^c$ denotes
the complement of $E_1$. Then $\limsup_N A_N/\max(0,\log\tr V_N)=\infty$ 
and hence there exists some $n_1\le 0$ such that
\[
\limsup_N \frac{1/(1+\alpha_N^*\cdot x_{n_1})}{\max(0,\log \tr V_N)} = \infty.
\]
Then there exits a subsequence of rounds $N_1 < N_2 < \cdots$  such that
\[
\lim_k \frac{1/(1+\alpha_{N_k}^*\cdot x_{n_1})}{\max(0,\log \tr V_{N_k})} =\infty.
\]
Because $\alpha_{N_k}^*\cdot x_{n_1} \rightarrow -1$ we have
\[
\limsup_N \frac{\frac{(\alpha_{N}^*\cdot x_{n_1})^2}{1+\alpha_{N}^*\cdot x_{n_1}}}
{\max(0,\log \tr V_N)} = \infty.
\]
If we compare this with  the left-hand side of (\ref{eq:empirical-bound-10-2}),
we see that a single term $x_{n_1}$ of the training data
contributes arbitrary large gain to
Skeptic in comparison to the right-hand side of (\ref{eq:boundphin}).
This implies that  $\limsup_N\log \cK_{1,N}^*=\infty$.
We have proved that by SOS Skeptic can weakly force $E_1$.
Therefore from now we only consider $\xi\in E_1$.

At this point we distinguish two cases 1) $E_2: \lim_N \tr V_N < \infty$
or 2) $E_2^c: \lim_N \tr V_N = \infty$.  Consider the first case and 
fix an arbitrary $\xi\in E_2\cap E_1$.  
For such a $\xi$ there exists $\delta(\xi) > 0$ such that
$\liminf_N (1+\alpha_N^* \cdot x_n) \ge \delta(\xi)$
for all $n<0$. Then
\[
V_{0,N}^* \le \frac{1}{\max(\epsilon_0,\delta(\xi))}\sum_{n=-n_0+1}^N x_n x_n\tp
\]
and hence the maximum eigenvalue $\lambda_{\max,0,N}$ of $V_{0,N}^*$ is bounded.
Then 
\[
s_{0,N}\tp V_{0,N}^{*-1} s_{0,N} \ge \frac{\Vert s_{0,N}\Vert^2}
{\lambda_{\max,0,N}}
\]
and
$\limsup_N \log \cK_{1,N}^* = \infty$ if 
$\limsup_N \Vert s_{0,N}\Vert^2 =\infty$.  Noting that
$\limsup_N \Vert s_{0,N}\Vert^2 =\infty$ if and only if 
$\limsup_N \Vert s_N\Vert^2 =\infty$,
we have shown that by SOS skeptic can weakly force
\[
\lim_N \tr V_N < \infty\  \Rightarrow  \limsup_N \Vert s_N\Vert^2 < \infty .
\]

Now consider the second case $E_2^c \cap E_1$.  
On $E_2^c \cap E_1$
\[
\lim_N \frac{\log \tr V_N}{\tr V_N}=0
\]
always holds.   Also on $E_2^c \cap E_1$
\[
\limsup_N \frac{\tr V_{0,0}(\alpha_N^*)}{\log \tr V_N}< \infty \quad \text{where}\ \ 
V_{0,0}(\alpha_N^*)=\sum_{n=-n_0+1}^{0} 
\frac{x_n x_n\tp}{1+\alpha_N^*\cdot x_n}.
\]
Therefore on $E_2^c \cap E_1$
\[
\lim_N \frac{\tr V_{0,0}(\alpha_N^*)}{\tr V_N}
=0.
\]
Also 
\[
\tr(V_{0,N}^* - V_{0,0}(\alpha_N^*)) \le \frac{1}{\epsilon_0} \tr V_N, 
\]
and hence on  $E_2^c \cap E_1$ 
\[
\limsup_N \frac{\tr V_{0,N}^*}{\tr V_N} 
= \limsup_N \frac{\tr V_{0,0}(\alpha_N^*) + \tr (V_{0,N}^*-V_{0,0}(\alpha_N^*))}
{\tr V_N} 
\le \frac{1}{\epsilon_0}.
\]

Now on the right-hand side of 
(\ref{eq:empirical-bound-10-2}),  
for every  $\xi\in E_2^c\cap E_1$ 
there exists $N_0=N_0(\xi)$  such that for all 
$n\ge N_0$ 
\[
\Phi_{0,N}(\alpha_N^*) \ge \frac{1}{2C_1} 
\frac{\Vert s_{0,N}\Vert^2}{\tr V_{0,N}^*}
\ge \frac{\epsilon_0}{4C_1} 
\frac{\Vert s_{0,N}\Vert^2}{\tr V_N}.
\]
Hence if  for this $\xi$ 
\[
\limsup_N \frac{\Vert s_{0,N} \Vert^2}{\tr V_{N} \log \tr V_{N}} = \infty
\]
then $\limsup \log \cK_{1,N}^*=\infty$ in view of Lemma \ref{lem:diff2}.
However on $E_2^c$ the following two events are equivalent:
\[
\limsup_N \frac{\Vert s_{0,N} \Vert^2}{\tr V_{N} \log \tr V_{N}} = \infty
\ \ \Leftrightarrow \ \ 
\limsup_N \frac{\Vert s_{N} \Vert^2}{\tr V_{N} \log \tr V_{N}} = \infty.
\]
This completes the proof of Theorem \ref{thm:gen-slln}.

\subsection{Better approximation to the capital process of SOS}
\label{subsec:approximation}

Note that (\ref{eq:delta-phi-upper-1}) is convenient because it gives an upper bound 
which always holds.
However bounding by $\Phi_n(\alpha_n^*)-\Phi_n(\alpha_{n-1}^*)\le 0$
in (\ref{eq:delta-phi-upper-0}) is not very accurate.
By expanding $\Phi_n(\alpha_{n-1}^*)$ at 
$\alpha = \alpha_n^*$ and by noting 
$\partial \Phi_n(\alpha_n^*) = 0$, we have
\begin{equation}
\label{eq:phin-approx}
\Delta \Phi_n = \frac{1}{2}\Delta \alpha_n^{*\textrm{t}} I_n({\bar \alpha}_n^*)
\Delta \alpha_n^*,
\quad {\bar \alpha}_n^* = \theta \alpha_{n-1}^* + 
(1-\theta)\alpha_n^*,\ 0 < \theta < 1,
\end{equation}
where $I_n(\alpha)=V_{0,n}(\alpha,\alpha)$ is a $d\times d$ positive-definite matrix 
given by
\begin{align*}
I_n(\alpha) = -\partial \partial\tp \Phi_n(\alpha) = 
\sum_{i=-n_0+1}^n x_i(\alpha)x_i(\alpha)\tp.
\end{align*}
Comparing (\ref{eq:phin-approx}) with 
the right-hand side of 
(\ref{eq:delta-phi-upper-0}), we see that the upper bound in
(\ref{eq:delta-phi-upper-0}) is about twice 
the actual value of $\Delta \Phi_n$.  Now 
by Lemma \ref{lem:convergence2}
and (\ref{eq:det-form2}),
we can approximate $\Delta \Phi_n$ as
\[
\Delta \Phi_n \sim  
\frac{1}{2}\log \frac{|I_n(\alpha_n^*)|}{|I_{n-1}(\alpha_{n}^*)|}.
\]
Then accumulated these sum is approximated as
\begin{align}
\label{eq:2-6}
\sum_{n=1}^N \Delta \Phi_n &\sim    
\frac{1}{2}\sum_{n=1}^N \log \frac{|I_n(\alpha_n^*)|}{|I_{n-1}(\alpha_{n}^*)|} = \frac{1}{2}\log\ [I_N],\quad  
[I_N] = \prod_{n=1}^N \frac{|I_n(\alpha_n^*)|}{|I_{n-1}(\alpha_{n}^*)|}.
\end{align}
Hence from (\ref{eq:2-3}), (\ref{eq:2-5}) and (\ref{eq:2-6}) we obtain
\begin{align*}
\log {\cal K}_{1,N}^* &= \log {\bar {\cal K}}_N^*-
\sum_{n=1}^N \Delta \Phi_n  - \Phi_{0,0}(\alpha_0^*)
\sim ND(g_N \| g^*_N) - \frac{1}{2}\log\ [I_N].
\end{align*}
The above result is summarized in the following theorem.
\begin{theorem}
The log capital of the sequential optimizing strategy
$\log {\cal K}_{1,N}^*$ is approximated  as
\begin{align}
\label{eq:2-7}
\log {\cal K}_{1,N}^* \sim   
ND(g_N \| g^*_N) - \frac{1}{2}\log\ [I_N],\quad [I_N] = \prod_{n=1}^N \frac{|I_n(\alpha_n^*)|}{|I_{n-1}(\alpha_{n}^*)|}.
\end{align}
\end{theorem}

Here we note that the quantity $|I_n(\alpha_n^*)|$ also appeared in the evaluation of Cover's universal portfolio \cite{cover:1991} in the name of sensitivity (curvature, volatility) index. Differently from the form 
$[I_N]$ in SOS, only the last term $|I_N(\alpha_N^*)|$ enters 
in the sensitivity index.
This difference reflects the fact that SOS depends on the intermediate moves of Reality's path $\xi^N = x_1\cdots x_N$, whereas the universal portfolio is independent of them.
 
We found that the approximation (\ref{eq:2-7}) is extremely accurate in practice 
(cf. Section \ref{sec:numerical}).  Thus 
we propose to
use this approximation as an information criterion for selecting betting items.
Let us denote the betting game with $d$ items by $Game(d)$, and suppose that there is a sequence of nested betting games such that
\begin{align*}
Game (1) \subset Game (2) \subset \cdots \subset Game ({\bar d}).
\end{align*}
We also write the main terms of (\ref{eq:2-7}) in $Game (d)$ as
\begin{align*}
\log {\cal K}_{1,N}^*(d) \sim  
ND_d(g_N \| g^*_N) - \frac{1}{2}\log\ [I_N]_d.
\end{align*}
As functions of $d$, $D_d(g_N \| g^*_N)$ increases monotonically
and $\log\ [I_N]_d$ is also expected to increase monotonically (cf. Section \ref{sec:numerical}). 
Hence due to the trade-off between $D_d(g_N \| g^*_N)$ and 
$\log\ [I_N]_d$ with respect to $d$, we can expect that  
$\max_{1\le d \le {\bar d}} \log {\cal K}_{1,N}^*(d)$ provides the optimal number $d^*$ of betting items. 
Including this subject, we will examine the obtained results by numerical examples in Section \ref{sec:numerical}.

Finally we give a brief proof of Theorem \ref{thm:gen-slln2}.
The point of the proof is to show that  $\alpha_N^* \rightarrow 0$ on $E'$.

\begin{proof}[Proof of Theorem \ref{thm:gen-slln2}]
$E'$ in (\ref{eq:min-log-max}) holds only if 
$\lminN \rightarrow\infty$.
Then by (\ref{eq:bound1}) and  (\ref{eq:empirical-bound-10-2}) we have
\[
\Phi_{0,N}(\alpha_N^*) \ge \frac{1}{C_1^2} \Vert \alpha_N^*\Vert^2 \lminN.
\]
Note that $\log |V_N| \le d \log \lmaxN$.
Therefore if $\limsup \Vert \alpha_N^*\Vert > 0$ then
$\limsup_N \cK_{1,N}^*=\infty$.  This shows that
conditional on $E'$ Skeptic can weakly force the event $\alpha_N^* \rightarrow 0$.

However when $\alpha_N^*\rightarrow 0$, for all sufficiently large $N$ 
we can approximate
\[
\Phi(\alpha_N^*)\sim \frac{1}{2} s_N\tp V_N^{-1} s_N, \qquad
\sum_{n=1}^N \Delta\Phi_n \sim \frac{1}{2} \log |V_N|.
\]
Since $\log |V_N| \rightarrow\infty$ on $E'$, if
$\limsup_N s_N\tp V_N^{-1} s_N/\log |V_N| > 1$ then 
$\limsup_N \log K_{1,N}^* = \infty$.  Therefore conditional on $E'$, by SOS Skeptic
can weakly force $\limsup_N s_N\tp V_N^{-1} s_N/\log |V_N| \le  1$.
\end{proof}


\section{High frequency limit order SOS in multiple asset trading games in continuous time}
\label{sec:high}

In this section we generalize the results of 
\cite{takeuchi-kumon-takemura-bernoulli} to the multi-dimensional case and 
apply SOS as a high-frequency limit order type investing strategy
to multiple asset trading games in continuous time. 
We follow the notation and the definitions in \cite{takeuchi-kumon-takemura-bernoulli}.
For simplicity of statements we make convenient assumptions 
and only present salient aspects of SOS.

Let $\Omega^d$ 
denote the set of $d$-dimensional (component-wise) positive continuous functions on 
$[0,\infty)$.  
Market (Reality) chooses an element $S(\cdot) \in \Omega^d$.
Investor (Skeptic) enters the market at time $t=t_0=0$  
with the initial capital of ${\cal K}(0)=1$ and he will buy or
sell any amount of the assets $S(t) = (S^1(t), \dots, S^d(t))\tp$ at discrete time points $0=t_0 <  t_1 < t_2 < \cdots $, provided that his capital always remains non-negative. 
His repeated tradings up to time $t_i$ determine
$M_i = (M_i^1, \dots, M_i^d)\tp \in {\mathbb R}^d$, where $M_i^j$ denotes the amount of the asset $S^j(t)$ he holds for the time interval $[t_i,t_{i+1})$. 
Let ${\cal K}(t)$ denote the capital of Investor 
at time $t$, which is written as
\begin{align}
\label{eq:1-2}
&{\cal K}(t) = {\cal K}(t_i) + M_i\cdot(S(t) - S(t_i))\quad 
\textrm{for}\ \ t_i \le t  < t_{i+1},
\end{align}
with ${\cal K}(0) = 1$. By defining
\[
\alpha_i = (\alpha_i^1, \dots, \alpha_i^d)\tp,\quad 
\alpha_i^j = \frac{M_i^jS^j(t_i)}{{\cal K}(t_i)}, 
\]
we rewrite (\ref{eq:1-2}) as
\[
{\cal K}(t) = {\cal K}(t_i)
\left(1 + \alpha_i\cdot \frac{S(t) - S(t_i)}{S(t_i)}\right)\quad 
\textrm{for}\ \ t_i \le  t < t_{i+1} 
\]
in terms of the returns of the assets given by
\begin{align*}
\frac{S(t) - S(t_i)}{S(t_i)} = \left(\frac{S^1(t) - S^1(t_i)}{S^1(t_i)}, \dots, \frac{S^d(t) - S^d(t_i)}{S^d(t_i)}\right)\tp.
\end{align*}

Investor takes some constant $\delta >0$ and decides the trading times $t_1, t_2, \dots$ by the ``limit order'' type strategy as follows. 
After $t_i$ is determined, 
let $t_{i+1}$ be the first time after $t_i$ when 
\begin{align}
\label{eq:1-5}
\left\|\frac{S(t_{i+1}) - S(t_i)}{S(t_i)}\right\| = \delta
\end{align}
happens. This process leads to a discrete time bounded forecasting game
embedded into the asset trading game in the following manner.  Let
\begin{align*}
x_n = (x_n^1, \dots, x_n^d)\tp \in C_\delta,\quad 
x_n^j = \frac{S^j(t_{n+1}) - S^j(t_n)}{S^j(t_n)},
\end{align*}
where $C_\delta$ denotes the sphere of radius $\delta$ in ${\mathbb R}^d$ 
given by (\ref{eq:1-5}), and also write ${\cal K}_n = {\cal K}(t_{n+1})$. 
Then we have the protocol of an embedded 
discrete time bounded forecasting game.

\bigskip\noindent
\textsc{Embedded Discrete Time Bounded Forecasting Game} \\
\textbf{Protocol:}

\parshape=6
\IndentI   \WidthI
\IndentI   \WidthI
\IndentII  \WidthII
\IndentII  \WidthII
\IndentII  \WidthII
\IndentI   \WidthI
\noindent
${\cal K}_0 :=1$, $\delta > 0$.\\
FOR  $n=1, 2, \dots$:\\
  Investor announces $\alpha_n\in{\mathbb R}^d$.\\
  Market announces $x_n\in C_\delta$.\\
  ${\cal K}_n = {\cal K}_{n-1}(1 + \alpha_n\cdot x_n)$.\\
END FOR\\
\

We now fix $T>0$, and 
Investor trades in the time interval $[0, T]$ by SOS  in (\ref{eq:2-3b}).
For $A>0$ let 
\[
E_{A,0,T} =  \{
S\in \Omega^d \  \mid  \ 
|\log S^j(x)-\log S^j(y)|\le A,\ \exists j\in \{1, \dots, d\},\ 
0 \le \forall x < \forall y  \le T\}.
\]
Market is assumed to choose $S(\cdot)\in  E_{A,0,T}^c$, 
which means that all $d$ items are active in some
time interval in $[0,T]$.
We define $N = N(T, \delta, S(\cdot))$ by 
$t_{N} < T \le t_{N+1}$. 
Note that by taking $\delta$ sufficiently
small, 
\begin{align*}
N(T, \delta, S(\cdot)) \ge  \frac{A}{\delta}
\end{align*}
for every $S(\cdot)\in E_{A,0,T}^c$, so that 
$N\rightarrow\infty$ as $\delta\rightarrow 0$.
Investor's capital ${\cal K}_\delta (T)$ at $t = T$ is written as
\begin{align*}
{\cal K}_\delta(T) 
= {\cal K}^*_{1,N}\left(1 + \alpha_{N-1}^*\cdot 
\frac{S(T) - S(t_{N})}{S(t_{N})}\right).
\end{align*} 
Since $\left\|\frac{S(T) - S(t_{N})}{S(t_{N})}\right\| 
\le \delta$, we have from (\ref{eq:2-7})
\begin{align}
\label{eq:3-1}
\log {\cal K}_\delta(T) = \log {\cal K}^*_{1,N} + O(1) 
\sim ND(g_{N}\| g_{N}^*) - \frac{1}{2}\log\ [I_N].
\end{align} 
The strategy (\ref{eq:2-8b}) is written as
$\alpha_{N}^* = \alpha_{N}^* (T,\delta, S(\cdot))=V_{0,N}^{*-1}s_{0,N}$.
We now assume (cf.\ Theorem \ref{thm:gen-slln2}) 
that  $\delta \alpha_{N}^* \rightarrow 0$ as $\delta\rightarrow 0$, i.e., 
Market chooses a path $S(\cdot) \in E_T'$, where
\[
E_T'=\{ S(\cdot)\in \Omega^d \mid 
\lim_{\delta\rightarrow 0} \delta \alpha_{N}^* (T,\delta, S(\cdot))=0 \}.
\]
Then 
\begin{align*}
\alpha_{N}^* 
= 
\Big(\sum_{n=-n_0+1}^{N} x_nx_n\tp \Big)^{-1}\sum_{n=-n_0+1}^{N}x_n \; (1+ 
O(\delta) )
= V_{0,N}^{-1}\Big(L(T) + \frac{1}{2}v_{0,N}\Big)(1+ O(\delta)),
\end{align*}
where
\begin{align*}
V_{0,N} &= \sum_{n=-n_0+1}^{N} x_n x_n\tp,\quad 
v_{0,N} = \Big(\sum_{n=-n_0+1}^{N}(x_n^1)^2, \dots, 
\sum_{n=-n_0+1}^{N}(x_n^d)^2\Big)\tp, \\
L(T) &= \log S(T) - \log S(0).
\end{align*}
We consider the first term $ND(g_{N}\| g_{N}^*)$ in (\ref{eq:3-1}).  
As was indicated by (\ref{eq:empirical-bound-10-2}),
\begin{align}
\label{eq:3-4}
ND(g_N \| g_N^*) &= \frac{1}{2}\alpha_N^{*\textrm{t}}V_{0,N}^* \alpha_N^*
(1 + O(\delta)) 
= \frac{1}{2}\alpha_N^{*\textrm{t}} V_{0,N} \alpha_N^* (1+ O(\delta)) \nonumber \\
&= \frac{1}{2}\Big[L(T)\tp V_{0,N}^{-1}L(T) + \frac{1}{2}\big(L(T)\tp V_{0,N}^{-1}v_{0,N} + v_{0,N}\tp V_{0,N}^{-1}L(T)\big) \nonumber \\
& \qquad\qquad\qquad
 + \frac{1}{4}v_{0,N}\tp V_{0,N}^{-1}v_{0,N}\Big](1 + O(\delta)).
\end{align}
The middle term is dominated by the first term and the third term by Cauchy-Schwarz:
\[
|L(T)\tp V_{0,N}^{-1}v_{0,N} + v_{0,N}\tp V_{0,N}^{-1}L(T)|
\le 2\sqrt{L(T)\tp V_{0,N}^{-1}L(T)} \sqrt{v_{0,N}\tp V_{0,N}^{-1}v_{0,N}}\ .
\]
Thus we consider the behavior of the first term and the third term.
Because $C_\delta$ is the sphere of radius $\delta$ we have
\[
\textrm{tr} V_N = \textrm{tr} D_N = N\delta^2 , 
\]
where 
\[
V_N = \sum_{n=1}^N x_nx_n\tp,\quad 
D_N = \textrm{diag}\ \Big(\sum_{n=1}^N (x_n^1)^2, \dots, 
\sum_{n=1}^N (x_n^d)^2 \Big).
\]
Also the training data are of order $\delta$. Hence 
$\textrm{tr} V_{0,N} - \textrm{tr} V_N= \textrm{tr} D_{0,N} -\textrm{tr}  D_N =O(\delta^2)$.

Let us decompose $V_{0,N}$ and $v_{0,N}$ as
\begin{align*}
&V_{0,N} = D_{0,N}^{1/2}R_{0,N}D_{0,N}^{1/2},\quad   
v_{0,N} = D_{0,N} 1_d, \quad 1_d = (1, \dots, 1)\tp,\\    
&D_{0,N} = \textrm{diag}\ \Big(\sum_{n=-n_0+1}^N (x_n^1)^2, \dots, 
\sum_{n=-n_0+1}^N (x_n^d)^2 \Big),
\end{align*}
where $R_{0,N}$ is the correlation matrix in $\{x^1_n, \dots, x^d_n\},\ n = -n_0+1, \dots, N$. Then 
\begin{align*}
v_{0,N}\tp V_{0,N}^{-1}v_{0,N} = 1_d\tp D_{0,N}^{1/2}R_{0,N}^{-1}D_{0,N}^{1/2}1_d
\ge \frac{1}{d} \textrm{tr} D_{0,N},
\end{align*}
because the maximum eigenvalue of $R_{0,N}$ is less than or equal to $d$.

Suppose that the H\"older exponent of $S(\cdot)$ is $0 < H < 1$ in the sense
that
\[
S(\cdot)\in E_{H,T} = \{ S(\cdot) \mid 0 < \liminf_{\delta\rightarrow 0}
\frac{\textrm{tr} V_N}{\delta^{(2-\frac{1}{H})}} \le 
\limsup_{\delta\rightarrow 0}
\frac{\textrm{tr} V_N}{\delta^{(2-\frac{1}{H})}} < \infty \}.
\]
By combining the arguments so far, if $S(\cdot)\in E_{A,0,T}^c\cap E_T'\cap 
E_{H,T}$  then the following implications hold:
\begin{align*}
&H > 0.5\ \Rightarrow\ \textrm{tr} D_N \to 0\ 
\Rightarrow\ L(T)\tp V_{0,N}^{-1}L(T) \to \infty,\\
&H < 0.5\ \Rightarrow\ \textrm{tr} D_N \to \infty\ 
\Rightarrow\ v_{0,N}\tp V_{0,N}^{-1}v_{0,N} \to \infty.
\end{align*}
Also it is easily shown that the second term $\frac{1}{2}\log [I_N]$ 
in (\ref{eq:3-1}) is of 
smaller order than $ND(g_{N}\| g_{N}^*)$. We summarize our result as a theorem, which is
a multi-dimensional generalization of Theorem 3.1 in  
\cite{takeuchi-kumon-takemura-bernoulli}.

\begin{theorem}
\label{thm:4-1}
By a high frequency $(\delta \to 0)$ limit order type 
sequential optimizing strategy in multiple asset trading games in continuous time,  Investor 
can essentially force $H = 0.5$ for $S(\cdot) \in E_{A,0,T}^c$ in the sense
\[
S(\cdot)\in E_{A,0,T}^c\cap E_T'\cap 
E_{H,T} \ \text{and}\ H \neq 0.5\ \ \Rightarrow\ \ {\cal K}_\delta(T) \to \infty\ \ 
\textrm{as}\ \ \delta \to 0.
\]
\end{theorem}

\section{Generality of high frequency limit order SOS}
\label{sec:generality}

In this section we show a generality of the high-frequency limit order SOS developed in the previous section, which implies that when
the asset price $S(t)$ follows the vector-valued geometric Brownian motion, our strategy automatically incorporates the well-known constant proportional betting strategy originated with Kelly (\cite{kelly}) and   
yields the likelihood ratio in the Girsanov's theorem for
geometric Brownian motion.  The convergence results in this section
are of measure-theoretic almost everywhere convergence.

When $S(t)$ is subject to the $d$-dimensional geometric Brownian motion with drift vector $\mu$ and non-singular volatility matrix $\sigma$,
\begin{align*}
L(T) = \bigg(\mu-\frac{1}{2}\sigma^2 \bigg)T + \sigma W(T),
\end{align*}
where $W(\cdot)$ denotes the $d$-dimensional standard Brownian motion, and $\sigma^2$ denotes the $d$-dimensional vector with the diagonal elements of $\sigma\sigma\tp$. 
In this section we let $T \rightarrow\infty$ and also let $\delta=\delta_T \rightarrow 0$
in such a way that $|\log \delta_T|= o(\sqrt{T})$.
We  have
\begin{align*}
V_{0,N} = (\sigma \sigma\tp)T (1 + O(\delta_T)),
\end{align*}
and hence we can evaluate
\begin{align}
\label{eq:5-1b}
\alpha_{N}^* = \Big[(\sigma \sigma\tp)^{-1}{\mu}+ 
\frac{(\sigma^{-1})\tp W(T)}{T}\Big](1 + O(\delta_T)).
\end{align}
The first term in the right-hand side of (\ref{eq:5-1b}) is the constant vector, which is derived also from the so-called Kelly criterion of maximizing  $E[\log {\cal K}(T)]$. 
 
Next consider $ND(g_{N}\| g_{N}^*)$ in (\ref{eq:3-1}), which was also   indicated by (\ref{eq:3-4}),
\begin{align*}
ND(g_N \| g_N^*) &= \frac{1}{2}\alpha_N^{*\textrm{t}}V_{0,N} \alpha_N^* 
(1 + O(\delta_T))\\
&= \Big[\frac{T}{2}\mu \tp (\sigma\sigma\tp)^{-1}\mu  
+ \frac{1}{2}\big((\sigma^{-1}\mu)\tp W(T) + W(T)\tp (\sigma^{-1}\mu)\big)\Big](1 + O(\delta_T)).
\end{align*}
The log capital (\ref{eq:3-1}) is then expressed as
\begin{align*}
\log {\cal K}_{\delta_T}(T) &= \Big[\frac{1}{2}\big((\sigma^{-1}\mu)\tp W(T) + W(T)\tp (\sigma^{-1}\mu)\big) + \frac{T}{2}\mu\tp (\sigma \sigma\tp)^{-1}\mu\\
&\qquad - \frac{1}{2}\log T + \log \delta_T\Big](1 + O(\delta_T)) + O(1).
\end{align*}
Hence 
the main terms on the right-hand side 
\begin{align*}
-\log {\cal K}(T) = -\frac{1}{2}\big((\sigma^{-1}\mu)\tp W(T) + W(T)\tp (\sigma^{-1}\mu)\big) - \frac{T}{2}\mu\tp (\sigma \sigma\tp)^{-1}\mu
 + o(\sqrt{T})
\end{align*}
provide the likelihood ratio of the unique martingale measure known as the Girsanov's theorem in multiple assets case, and we obtain
\begin{align}
\label{eq:5-1}
\lim_{T\to \infty}\frac{\log {\cal K}(T)}{T} = 
\frac{1}{2}\mu\tp (\sigma \sigma\tp)^{-1}\mu.
\end{align}

Finally we discuss mutual information quantities  among subgames of the
multi-dimensional bounded forecasting game.
Let us denote the quadratic form in the right-hand side of (\ref{eq:5-1}) by
\begin{align}
\label{eq:5-2}
Q(S) = Q(S^1, \dots, S^d) = \frac{1}{2}\mu\tp (\sigma \sigma\tp)^{-1}\mu,
\end{align}
which designates the optimal exponential growth rate of Investor's capital process with $d$ joint betting items $S = (S^1, \dots, S^d)$. We partition $S$ into the following form
\begin{align*}
S_{[1]} = (S^{j_1}, \dots, S^{j_{k_1}}),\ 
S_{[2]} = (S^{j_{k_1+1}}, \dots, S^{j_{k_2}}),\ \dots,\  
S_{[m]} = (S^{j_{k_{m-1}+1}}, \dots, S^{j_{k_m}}),
\end{align*}
and assume that Investor is allowed to trade the above $m$ groups of joint sub-betting items successively during the one period of the $d$ joint trading. 
Then the corresponding optimal exponential growth rate of Investor's capital process becomes 
\begin{align}
\label{eq:5-3}
Q(S_{[1]}) + Q(S_{[2]}) + \cdots + Q(S_{[m]}).
\end{align}
Note that among (\ref{eq:5-2}) and (\ref{eq:5-3}) for all possible partitions
there is no general dominance relations
and this
argument leads to the notion of mutual information quantity between betting games, which will be treated in a forthcoming paper.

\section{Numerical examples}
\label{sec:numerical}

In this section we give some numerical examples on the stock price
data from the Tokyo Stock Exchange. The data are daily closing 
prices from January 4th in 2000 to March 31st in 2006 for several Japanese companies listed on the first section of the TSE. There are $T = 1536$ daily closing prices.

{}From this data we construct the bounded forecasting game in the following manner. At first the daily returns  
$s_t^j = (S_{t+1}^j - S_t^j)/S_t^j,\ t = 1, \dots, T-1,\ j = 1, \dots, d$ of $d$ items are transformed to $[-1,1]$ by
\begin{align*}
z_t^j = \frac{2s_t^j - {\bar s}_t^j - \underline{s}_t^j}
{{\bar s}_t^j - \underline{s}_t^j}\in [-1, 1],\quad {\bar s}_t^j = 
\max_{1\le t \le T-1}s_t^j,\ \ 
\underline{s}_t^j = \min_{1\le t \le T-1}s_t^j.
\end{align*}
Next $2^d$ training data $\tilde{z}_t = (\pm 1, \dots,\pm1)\tp,\ t = 1, \dots 2^d$, and a forecasting time $F = cT,\ 0 < c < 1$ are  prepared, and forecasting value
for the $j$-th component is 
\begin{align*}
\rho^j = \frac{1}{2^d+F}\Big(\sum_{t=1}^{2^d} \tilde{z}_t^j
+ \sum_{t=1}^F z_t^j\Big),\quad j = 1, \dots, d.
\end{align*}
Then the bounded variables $x_n = (x_n^1, \dots, x_n^d)\tp$ in the protocol are introduced as 
\begin{align*}
x_n^j = 
\begin{cases}
\tilde{z}_n^j - \rho^j, & 1 \le n \le 2^d \\
z_{n-2^d+F}^j - \rho^j, & 2^d < n \le N = 2^d+T-1-F.\ 
\end{cases}
\end{align*}

Figures 1-5 and Figures 6-10 exhibit the cases of three items Takeda, Toyota, Kirin with $F = 0.17T$ and $F = 0.25T$, respectively. The notations in the figures are as follows and their final values at the end of round $N$ are indicated in the figures.
\begin{align*}
&K_n^0 = \bar{{\cal K}}_n^* = \exp (nD(g_n\|g_n^*)),\quad 
K_n^1 = {\cal K}_n^*,\quad 
K_n^2 = \frac{\bar{{\cal K}}_n^*}{\sqrt{[I_n]}},\\
&LK_n^0 = \log \bar{{\cal K}}_n^* = nD(g_n\|g_n^*),\quad 
LK_n^1 = \log {\cal K}_n^*,\quad 
LK_n^2 = nD(g_n\|g_n^*) - \frac{1}{2}\log\ [I_n],\\
&LD_n^1 = \log \bar{{\cal K}}_n^* - \log {\cal K}_n^*,\quad 
LD_n^2 = \frac{1}{2}\log\ [I_n],\quad 
LD_n^3 = \frac{3}{2}\log n,\\
&GR_n = D(g_n\|g_n^*),\quad 
QR_n = 
\frac{1}{2}{\bar x}_n\tp {\bar V}_n^{*-1}{\bar x}_n,\quad 
DR_n = \frac{\log\ [I_n]}{2n}.
\end{align*}
As suggested in Section \ref{subsec:approximation}, $K_n^1$ and $K_n^2$,
$LK_n^1$ and $LK_n^2$, $LD_n^1$ and $LD_n^2$ are almost overlapped in the figures. We can also see that the actual log deficiency 
$LD_n^1$ or $LD_n^2$ is far less than $LD_3$ which is the typical log  deficiency in the case of finite items such as in the horse race game. 
Furthermore Figures 5,10 show that the deficiency rate process $DR_n$ gives the precise convergence border rate for the growth rate process $GR_n$ or its approximated quadratic rate process $QR_n$. 
 
Figures 11-16 illustrate the cases of composite games
\begin{align*} 
Game (1) \subset Game (2) \subset Game (3) \subset Game (4) \subset Game (5)
\end{align*}
with five items 1.\ Takeda, 2.\ Toyota, 3.\ Kirin, 4.\ Tepco, 5.\ NNK in this order. 
As expected the following trade-off can be seen in the figures. 
\begin{align*}
&LK_n^0\ :\ G(1)\  < G(2)\  < G(3)\  < G(4)\  < G(5)\ ,\\
&LD_n^2\ :\ G(1)\ < G(2)\  < G(3)\  < G(4)\  < G(5)\ ,\\
&LK_n^1\ :\ G(1)\  < G(5)\  < G(2)\  < G(4)\  < G(3)\ .
\end{align*}
Hence the choice of the three items 1.\ Takeda, 2.\ Toyota, 3.\ Kirin 
is the most profitable one in the above composite games. 

Figures 17-20 compare the sequential optimizing strategy with the universal portfolio for one item Takeda, Toyota, Kirin, an imaginary data, respectively. The universal portfolio in its simplest form with one item can be performed in the following way. 

Divide the closed interval $A = \{\alpha \in \mathbb{R} \mid 
1+\alpha x \ge 0, \forall x\in D\}$
of prudent strategies 
into disjoint subintervals $A_1,\dots, A_M$. Then for the $m$-th
account $A_m$ with the initial capital ${\cal K}_0^{(m)} = 1/M$,
Skeptic continues the game with constant betting ratio $\alpha_m \in
A_m,\ m = 1, \dots, M$. His capital at the end of round $n$ is
expressed as ${\cal K}_n^U = \sum_{m=1}^M {\cal K}_n^{(m)}$. The
figures are the cases with $M = 100$ and the notations are
\begin{align*}
&K_n^{U0} = {\cal K}_n^U\ \textrm{without the training data}\ \{-1,\ 1\},\\
&K_n^{U1} = {\cal K}_n^U\ \textrm{with the training data}\ \{-1,\ 1\}.
\end{align*}
Figures 20-22 show the case of an imaginary data given by
\begin{align*}
x_{-1} = -1,\quad  x_0 = 1,\quad x_n = \frac{1}{n+1},\quad n = 1, \dots, 2000.
\end{align*}
In this case $LK_n^1 \sim a\log n - c,\ 0 < a < 1,\ c > 0$, which contrasts with the case of coin-tossing game  
$LK_n^1 \sim nD({\bar x}_n\|\rho) - \frac{1}{2}\log n$.

Figures 17-20 suggest that there is no general superiority between the sequential optimizing strategy and the universal portfolio. 

\begin{figure}[htbp]
\begin{minipage}{.5\linewidth}
\begin{flushleft}
Figure 1 : Closing prices of Takeda, Toyota, \\
\qquad \qquad \quad Kirin\quad $F = 0.17T$\\
Figure 2 : Capital processes\\
\qquad \qquad \qquad \qquad $K_n^0,\ K_n^1,\ K_n^2$\\
Figure 3 : Log capital processes\\
\qquad \qquad \qquad \qquad $LK_n^0,\ LK_n^1,\ LK_n^2$\\
Figure 4 : Log deficiency processes\\
\qquad \qquad \qquad \qquad $LD_n^1,\ LD_n^2,\ LD_n^3$\\
Figure 5 : Rate processes \\
\qquad \qquad \qquad \qquad $GR_n,\ QR_n,\ DR_n$
\end{flushleft}
\end{minipage}
\begin{minipage}{.5\linewidth}
\begin{center}
\includegraphics[width=8cm,height=7cm]{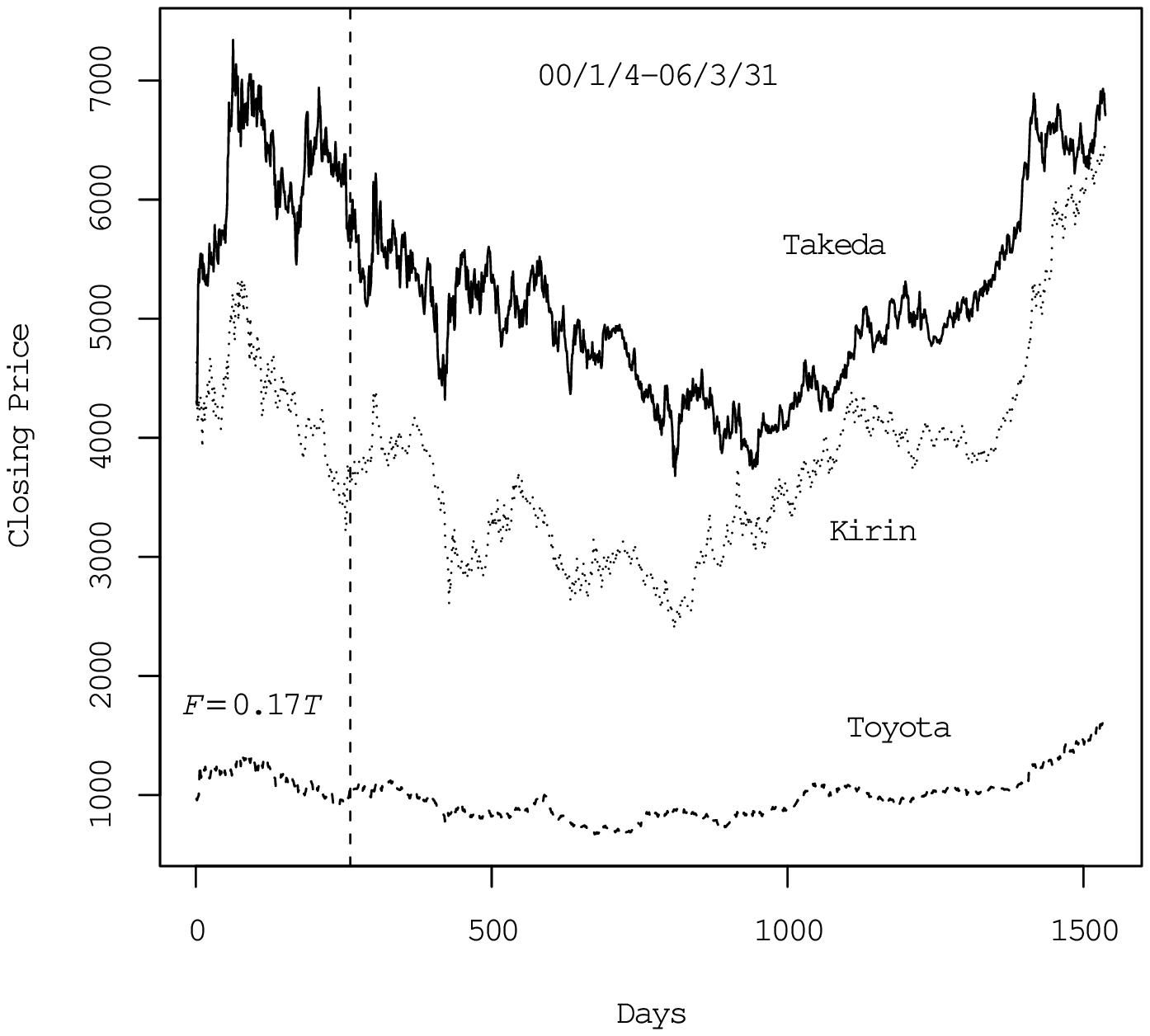}
\vspace*{-11mm}   
\caption{Closing prices}
\label{fig:1-1}
\end{center}
\end{minipage}
\begin{minipage}{.5\linewidth}
\begin{center}
\includegraphics[width=8cm,height=7cm]{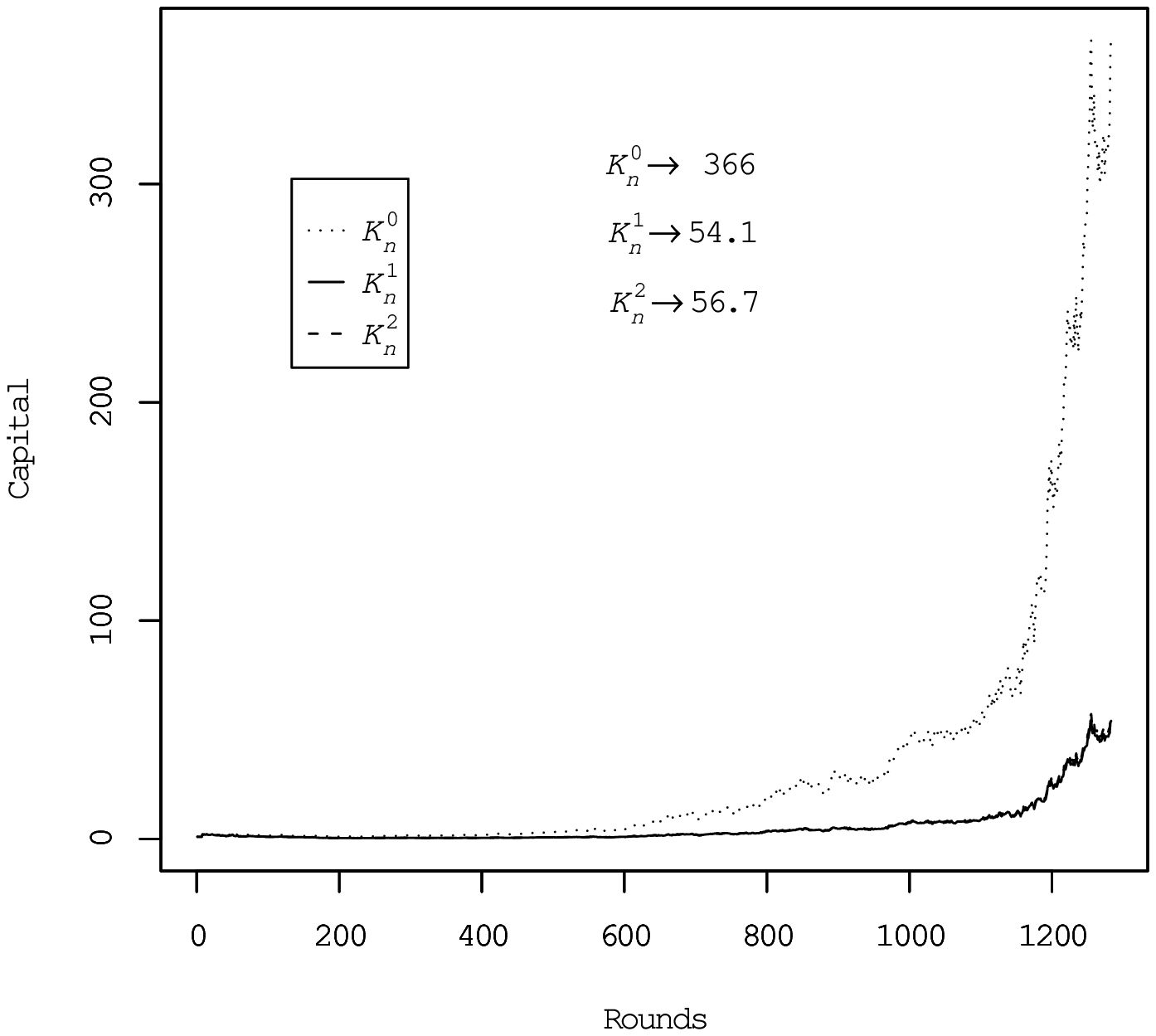}
\vspace*{-11mm}   
\caption{Capital processes}
\label{fig:1-2}
\end{center}
\end{minipage}
\begin{minipage}{.5\linewidth}
\begin{center}
\includegraphics[width=8cm,height=7cm]{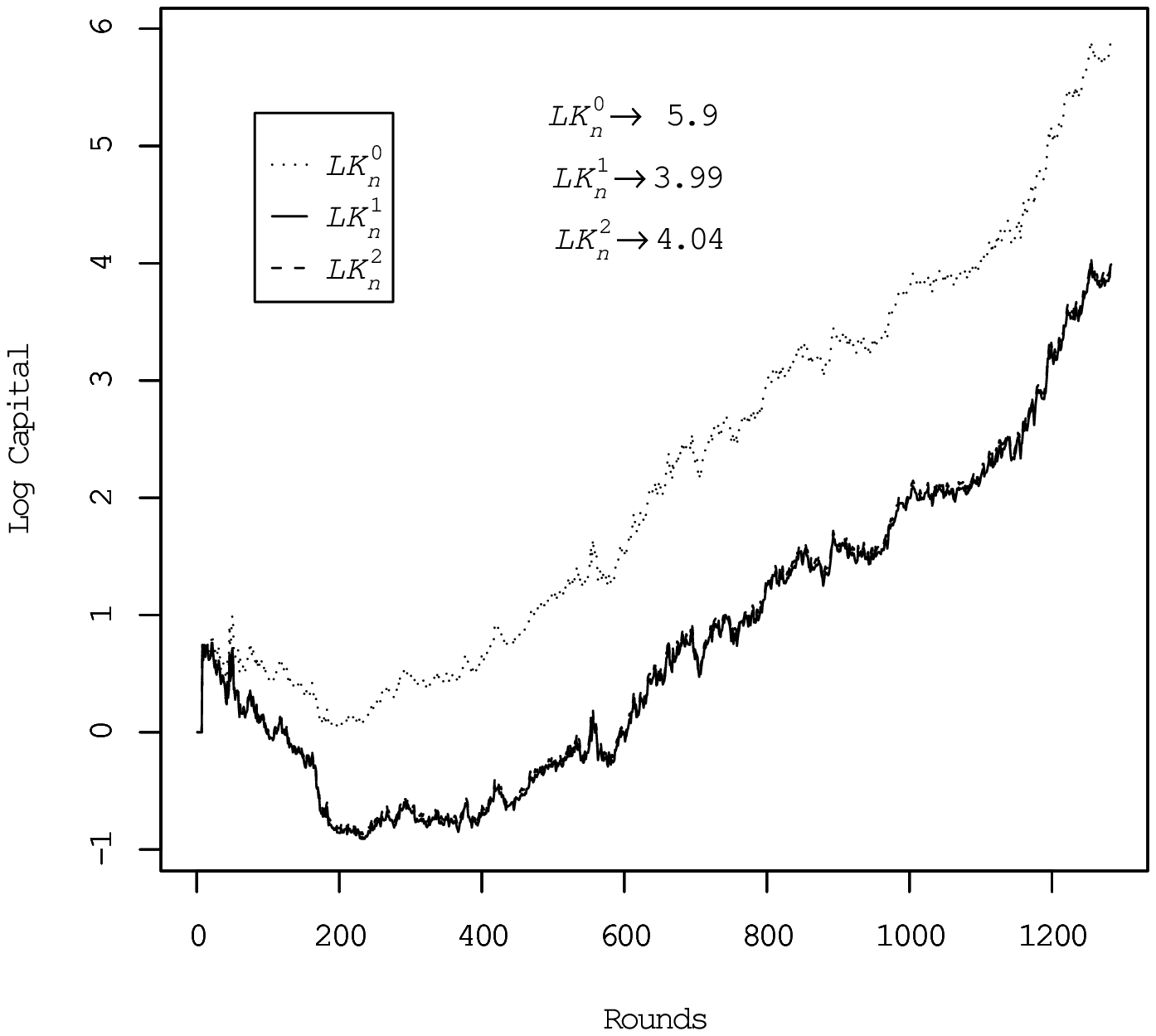}
\vspace*{-11mm}   
\caption{Log capital processes}
\label{fig:1-3}
\end{center}
\end{minipage}
\begin{minipage}{.5\linewidth}
\begin{center}
\includegraphics[width=8cm,height=7cm]{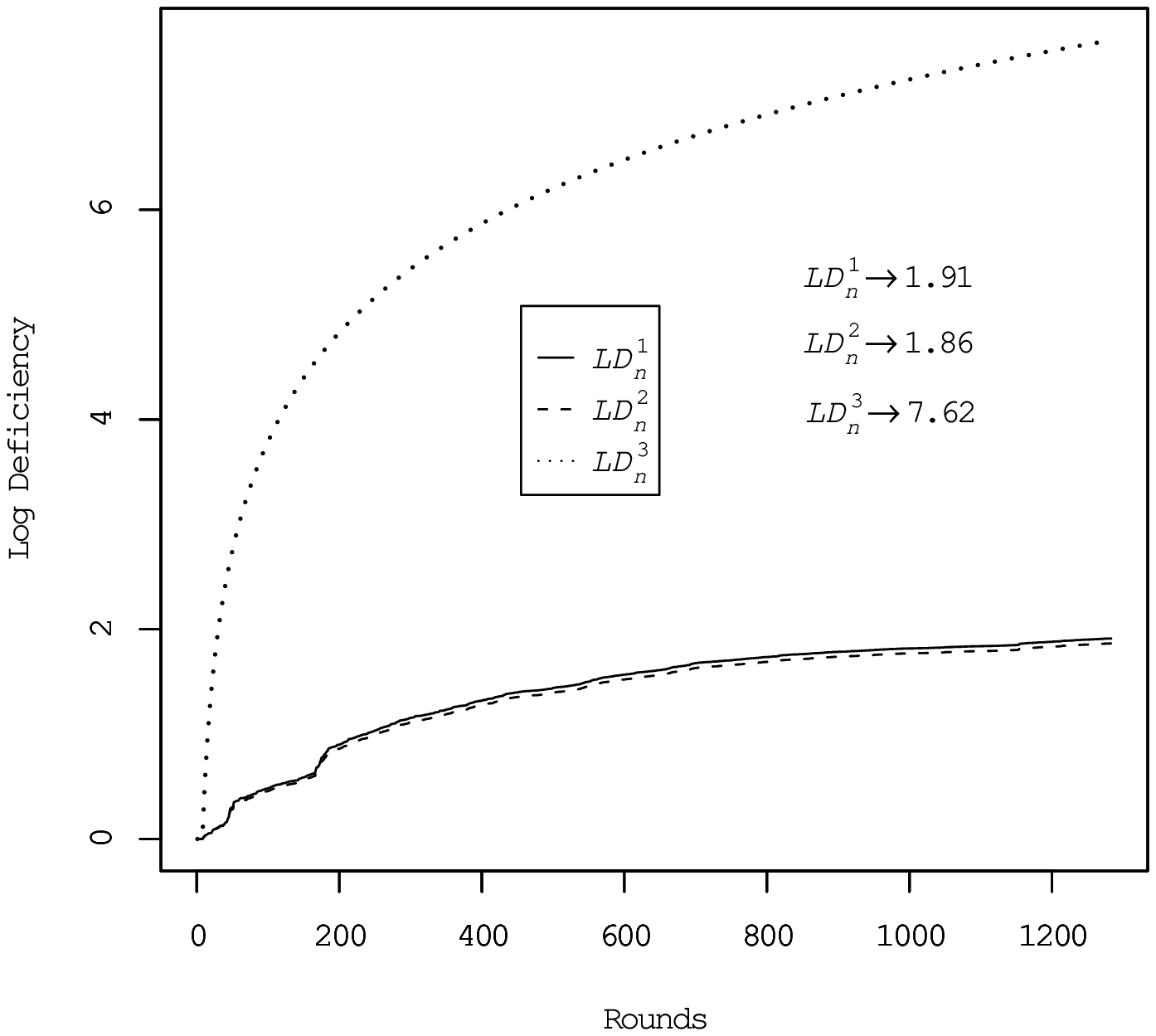}
\vspace*{-11mm}   
\caption{Log deficiency processes}
\label{fig:1-4}
\end{center}
\end{minipage}
\begin{minipage}{.5\linewidth}
\begin{center}
\includegraphics[width=8cm,height=7cm]{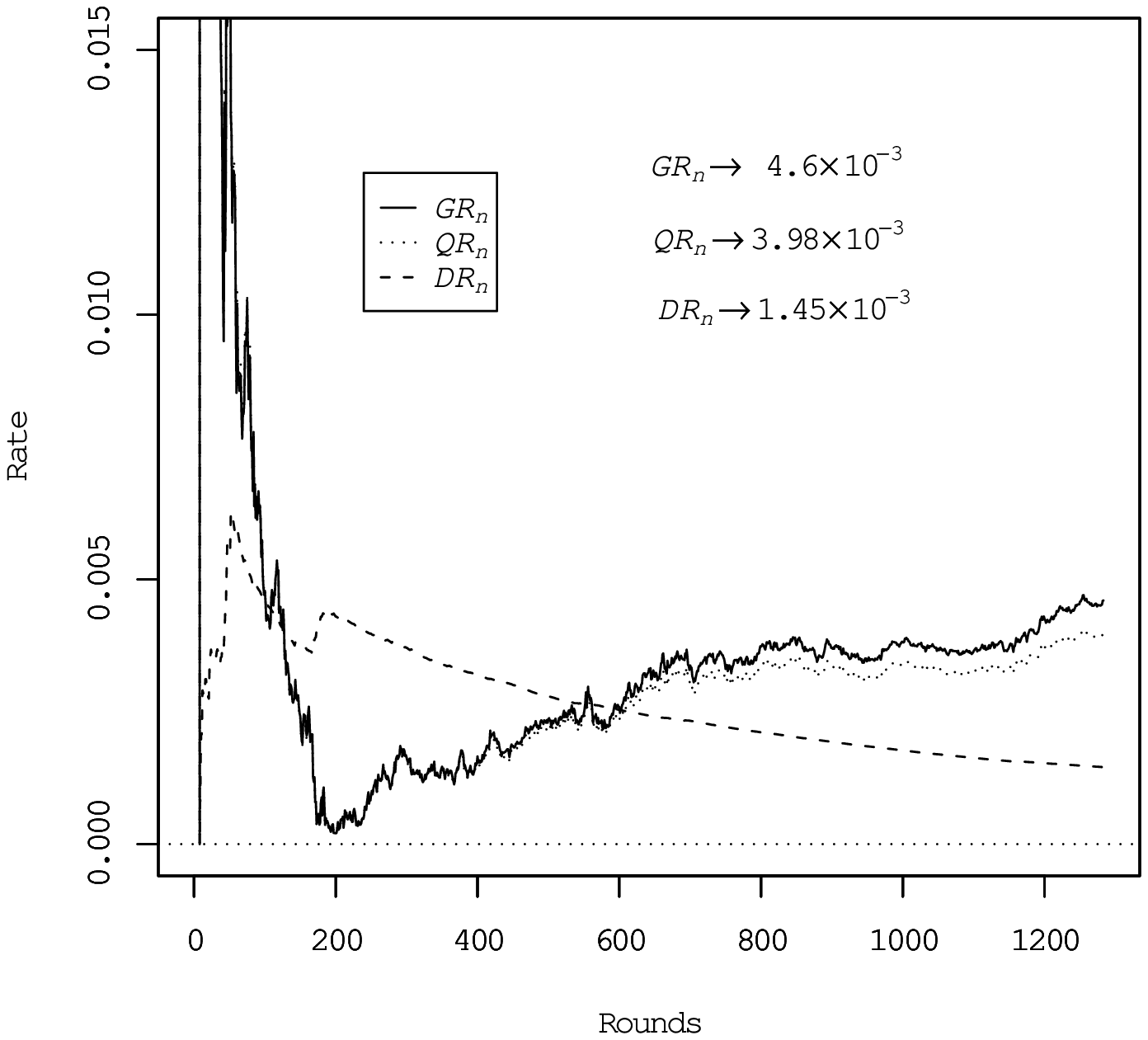}
\vspace*{-11mm}   
\caption{Rate processes}
\label{fig:1-5}
\end{center}
\end{minipage}
\end{figure}


\begin{figure}[htbp]
\begin{minipage}{.5\linewidth}
\begin{flushleft}
Figure 6 : Closing prices of Takeda, Toyota, \\
\qquad \qquad \quad Kirin\quad $F = 0.25T$\\
Figure 7 : Capital processes\\
\qquad \qquad \qquad \qquad $K_n^0,\ K_n^1,\ K_n^2$\\
Figure 8 : Log capital processes\\
\qquad \qquad \qquad \qquad $LK_n^0,\ LK_n^1,\ LK_n^2$\\
Figure 9 : Log deficiency processes\\
\qquad \qquad \qquad \qquad $LD_n^1,\ LD_n^2,\ LD_n^3$\\
Figure 10 : Rate processes \\
\qquad \qquad \qquad \qquad $GR_n,\ QR_n,\ DR_n$
\end{flushleft}
\end{minipage}
\begin{minipage}{.5\linewidth}
\begin{center}
\includegraphics[width=8cm,height=7cm]{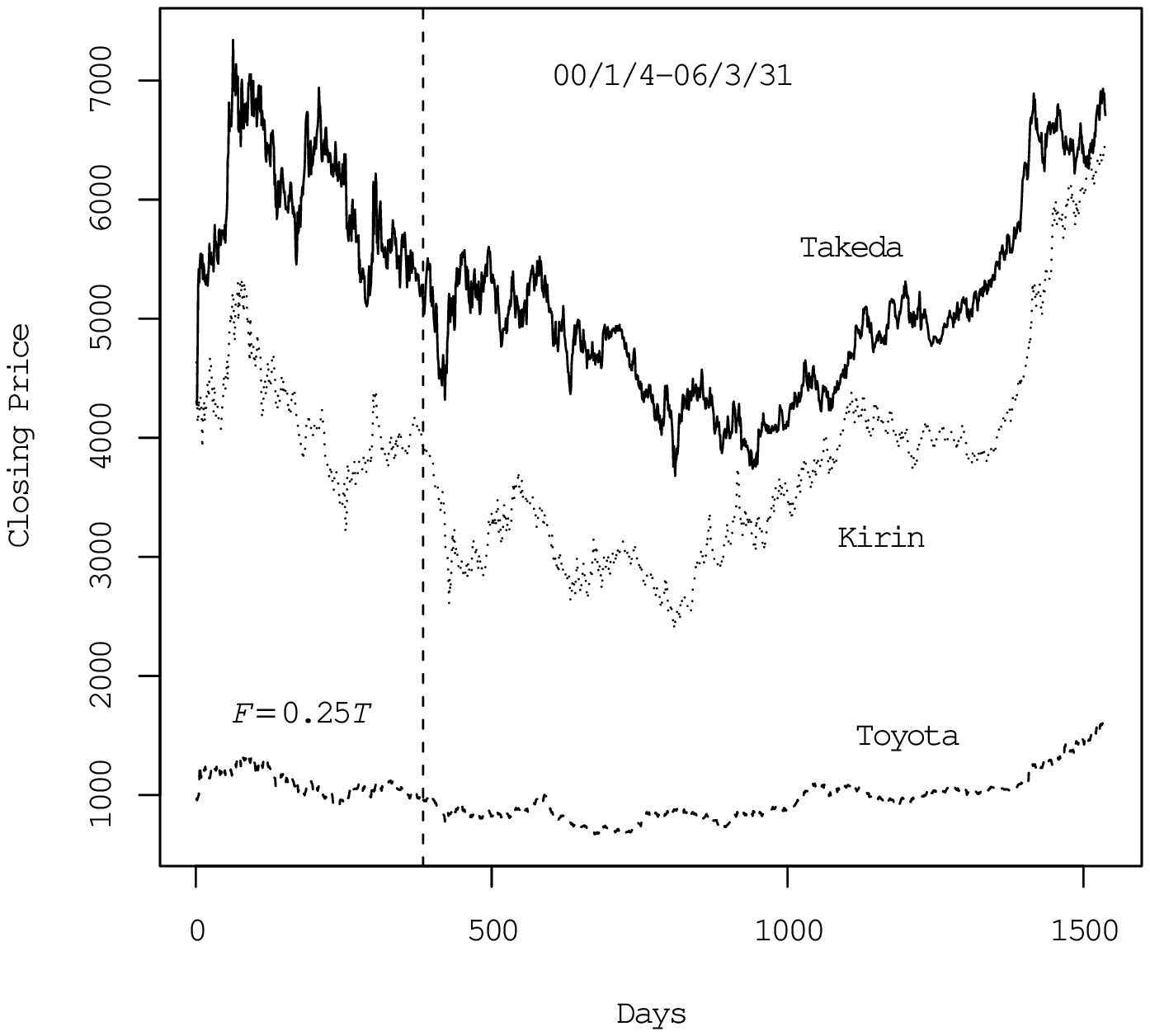}
\vspace*{-11mm}   
\caption{Closing prices}
\label{fig:1-6}
\end{center}
\end{minipage}
\begin{minipage}{.5\linewidth}
\begin{center}
\includegraphics[width=8cm,height=7cm]{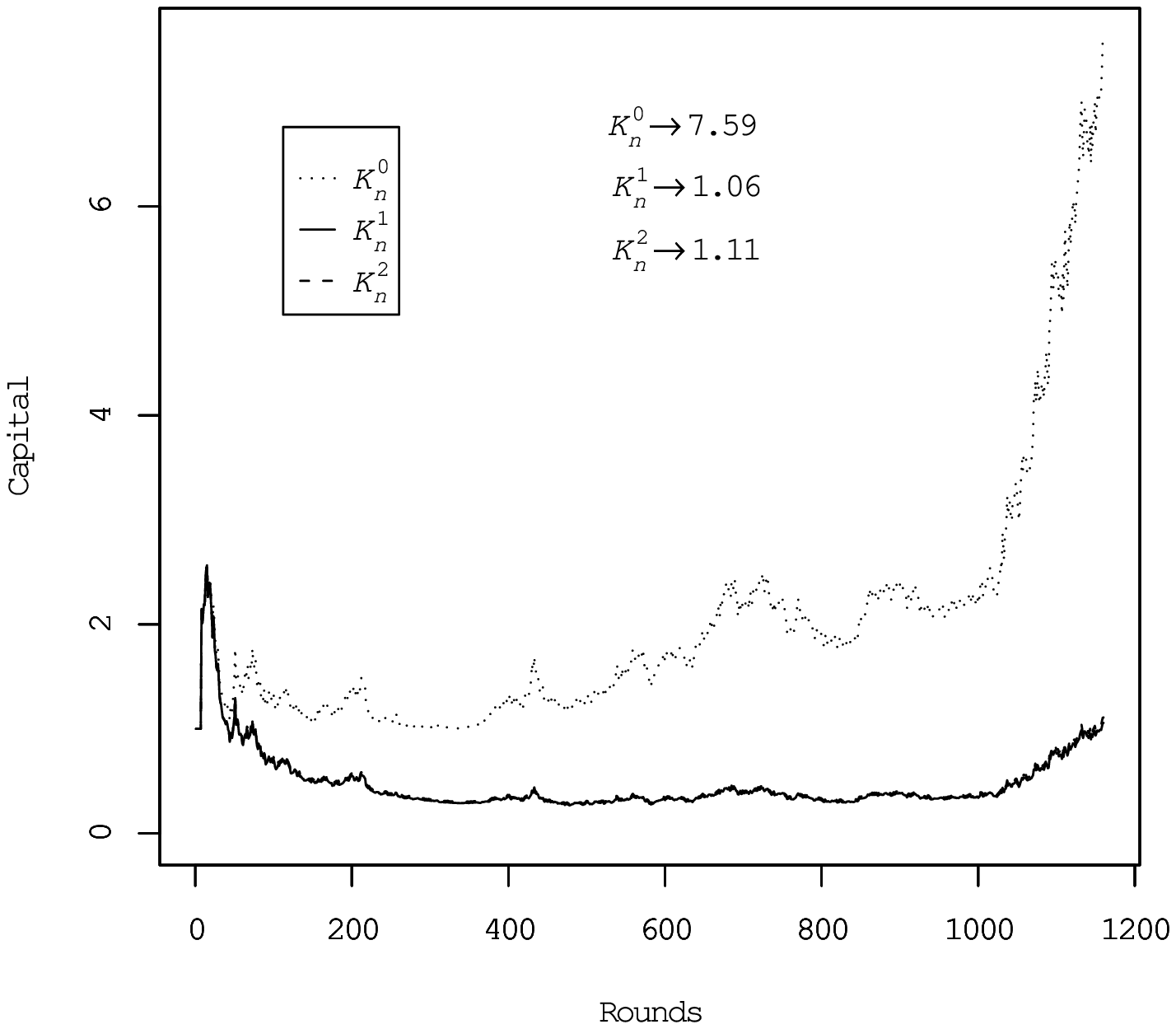}
\vspace*{-11mm}   
\caption{Capital processes}
\label{fig:1-7}
\end{center}
\end{minipage}
\begin{minipage}{.5\linewidth}
\begin{center}
\includegraphics[width=8cm,height=7cm]{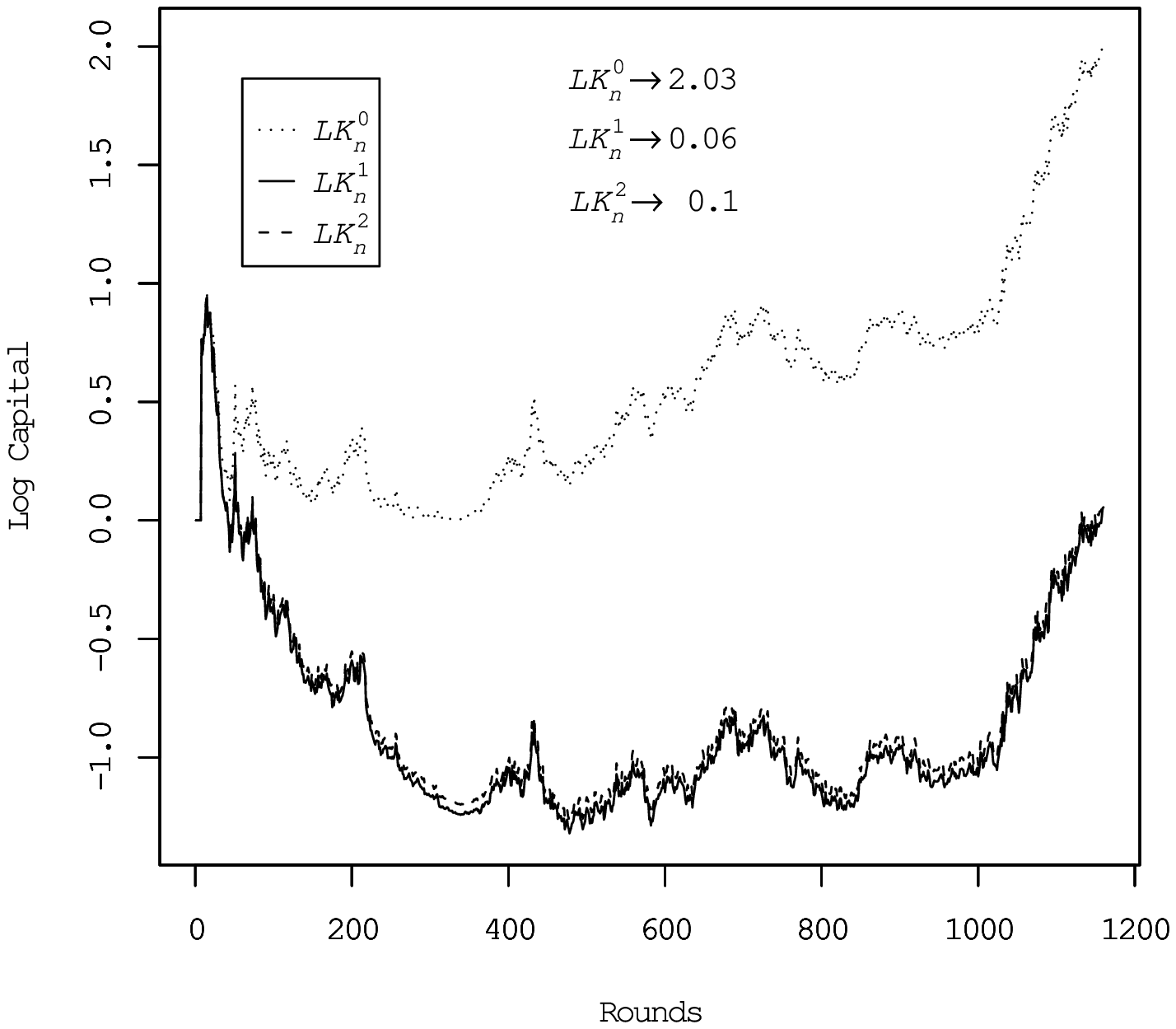}
\vspace*{-11mm}   
\caption{Log capital processes}
\label{fig:1-8}
\end{center}
\end{minipage}
\begin{minipage}{.5\linewidth}
\begin{center}
\includegraphics[width=8cm,height=7cm]{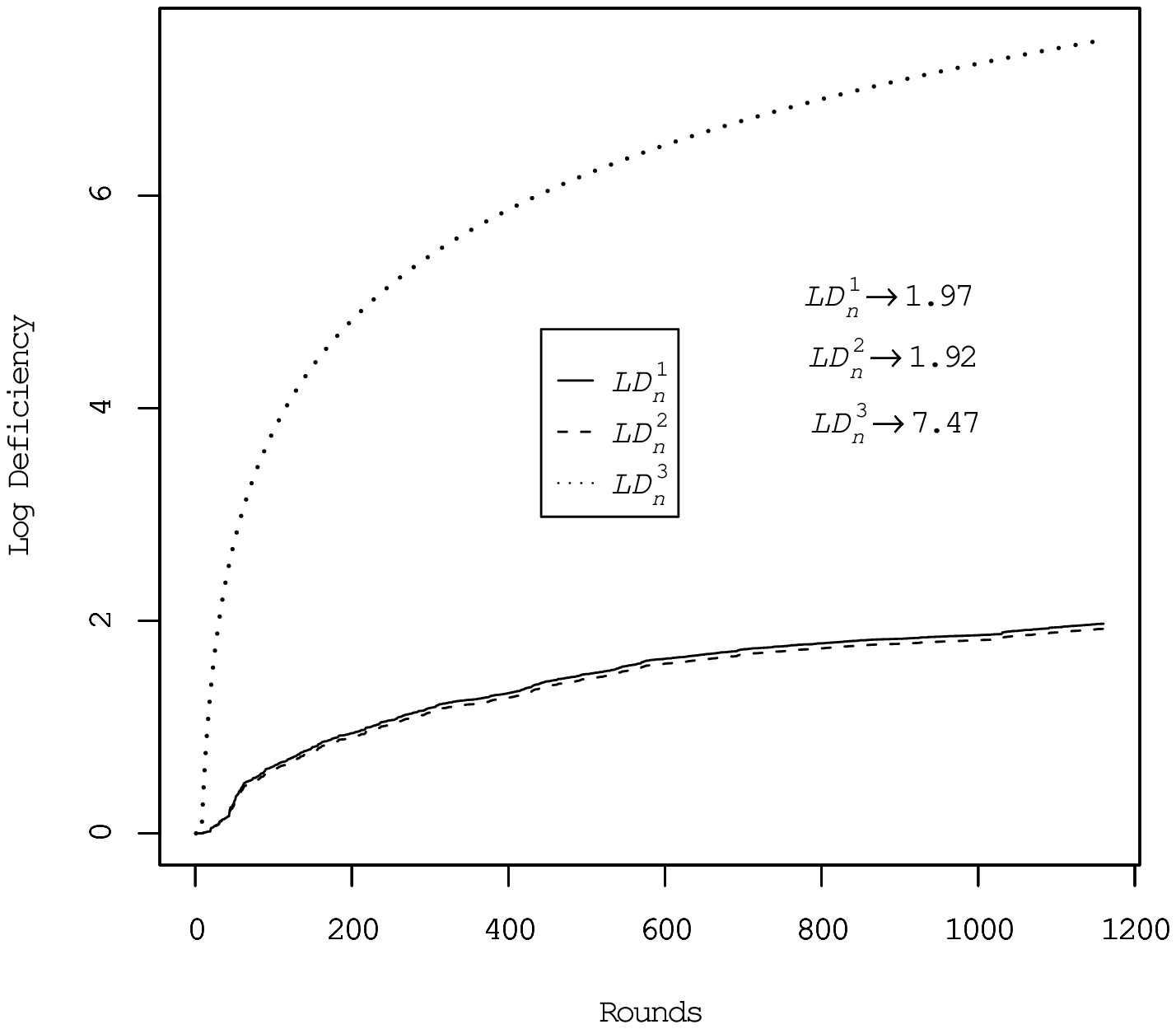}
\vspace*{-11mm}   
\caption{Log deficiency processes}
\label{fig:1-9}
\end{center}
\end{minipage}
\begin{minipage}{.5\linewidth}
\begin{center}
\includegraphics[width=8cm,height=7cm]{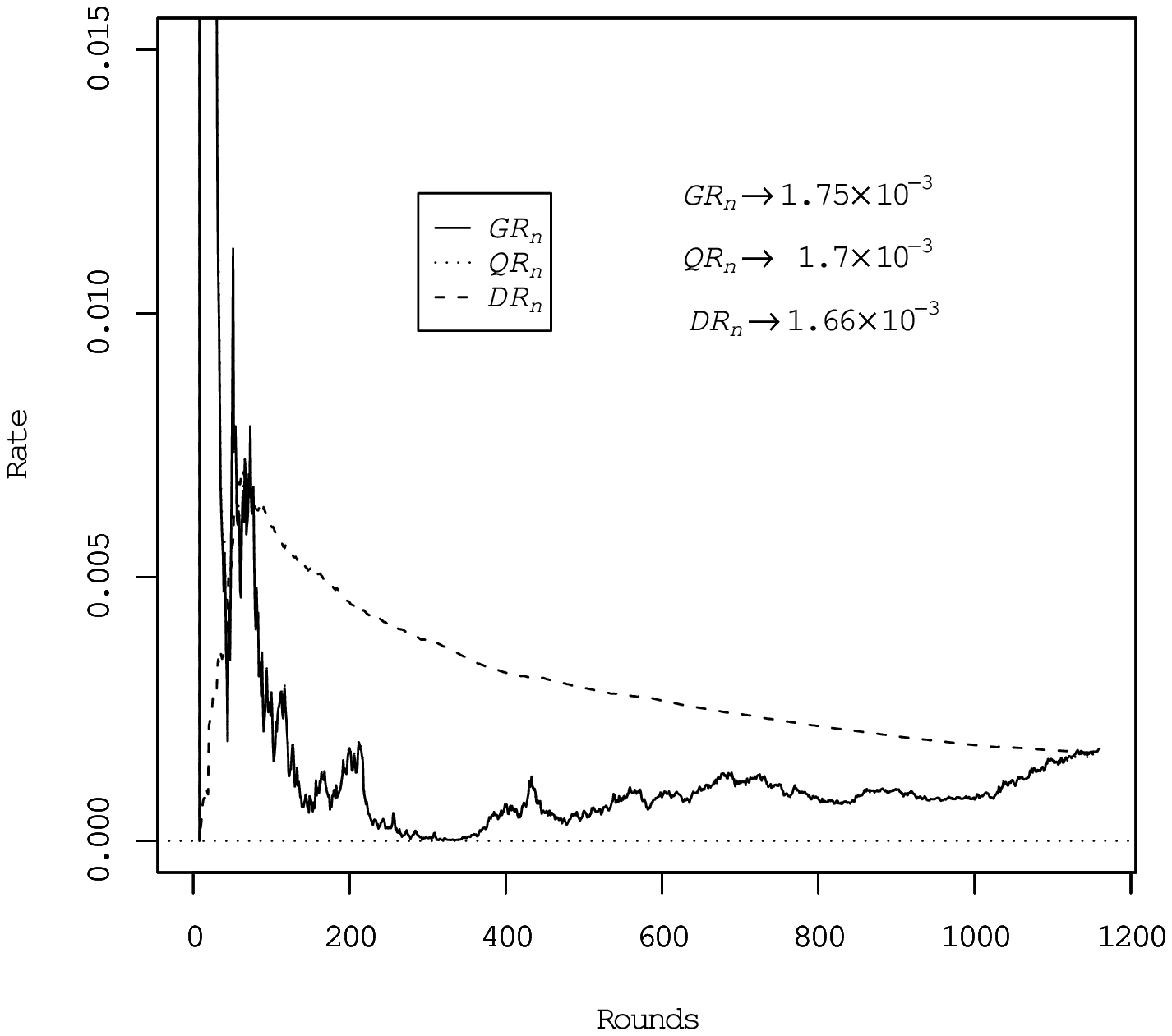}
\vspace*{-11mm}   
\caption{Rate processes}
\label{fig:1-10}
\end{center}
\end{minipage}
\end{figure}


\begin{figure}[htbp]
\begin{minipage}{.5\linewidth}
\begin{center}
\includegraphics[width=8cm,height=7cm]{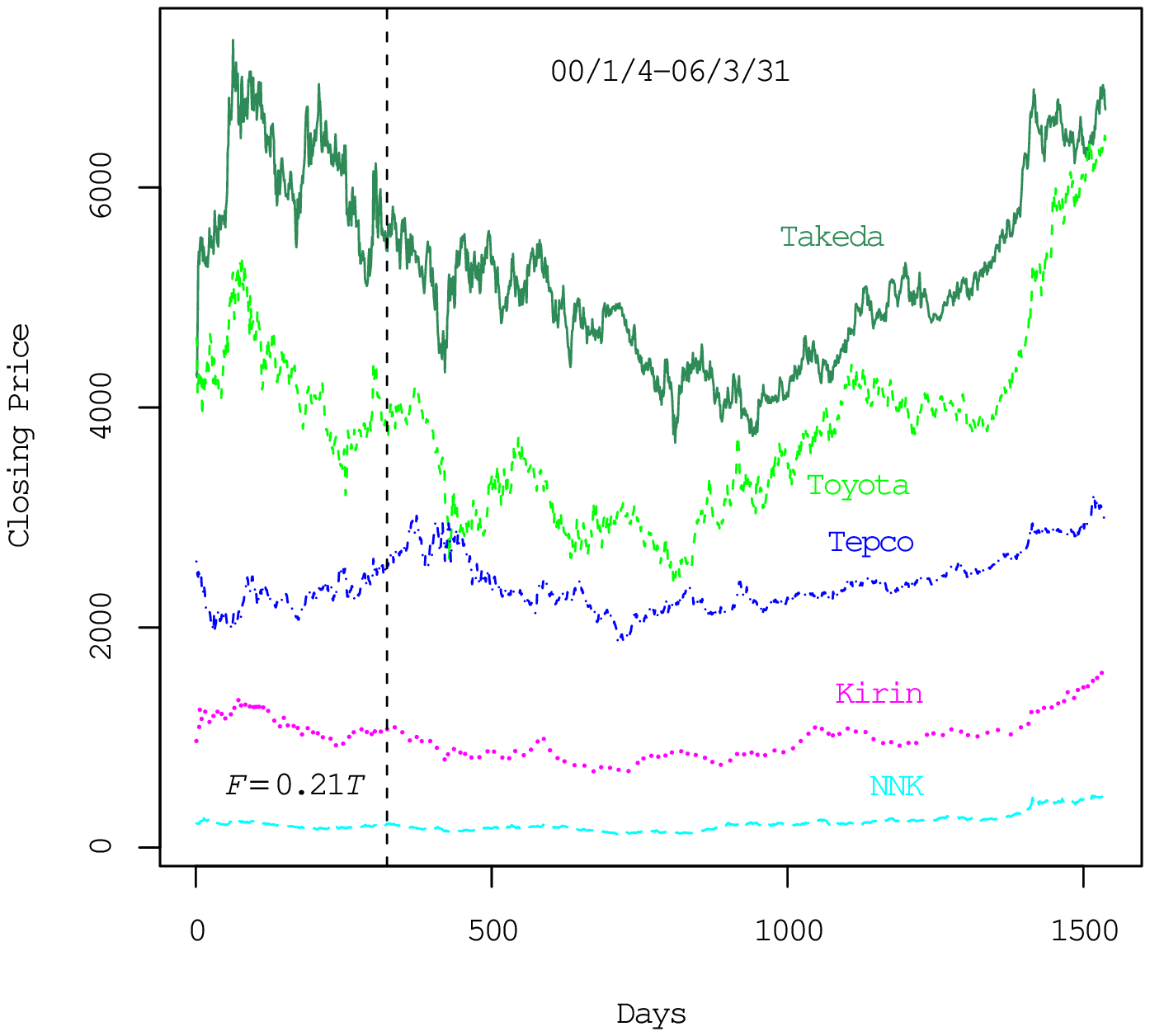}
\vspace*{-11mm}   
\caption{Closing prices}
\label{fig:1-11}
\end{center}
\end{minipage}
\begin{minipage}{.5\linewidth}
\begin{center}
\includegraphics[width=8cm,height=7cm]{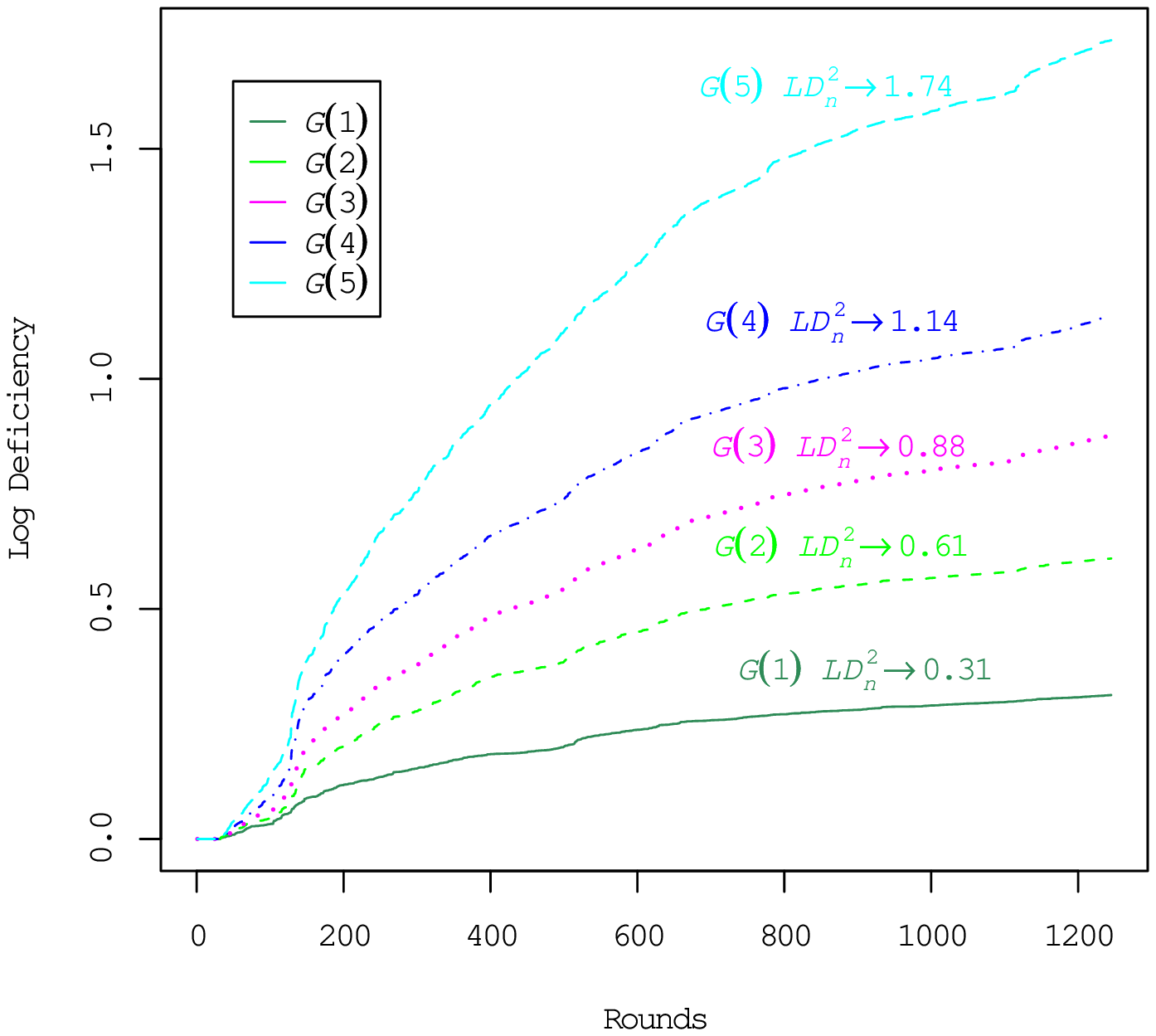}
\vspace*{-11mm}   
\caption{Log deficiency processes $LD_n^2$}
\label{fig:1-12}
\end{center}
\end{minipage}
\begin{minipage}{.5\linewidth}
\begin{center}
\includegraphics[width=8cm,height=7cm]{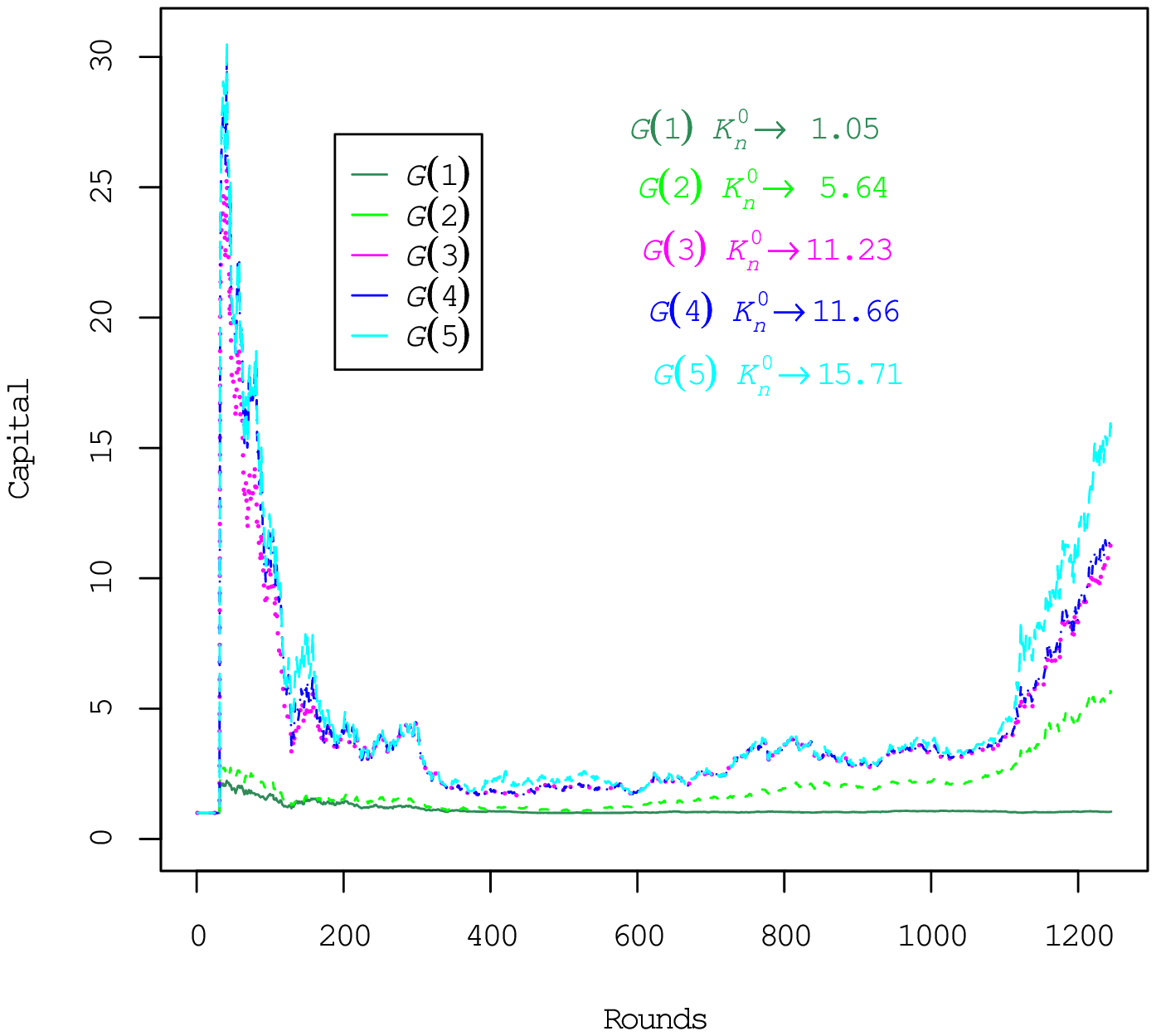}
\vspace*{-11mm}   
\caption{Capital processes $K_n^0$}
\label{fig:1-13}
\end{center}
\end{minipage}
\begin{minipage}{.5\linewidth}
\begin{center}
\includegraphics[width=8cm,height=7cm]{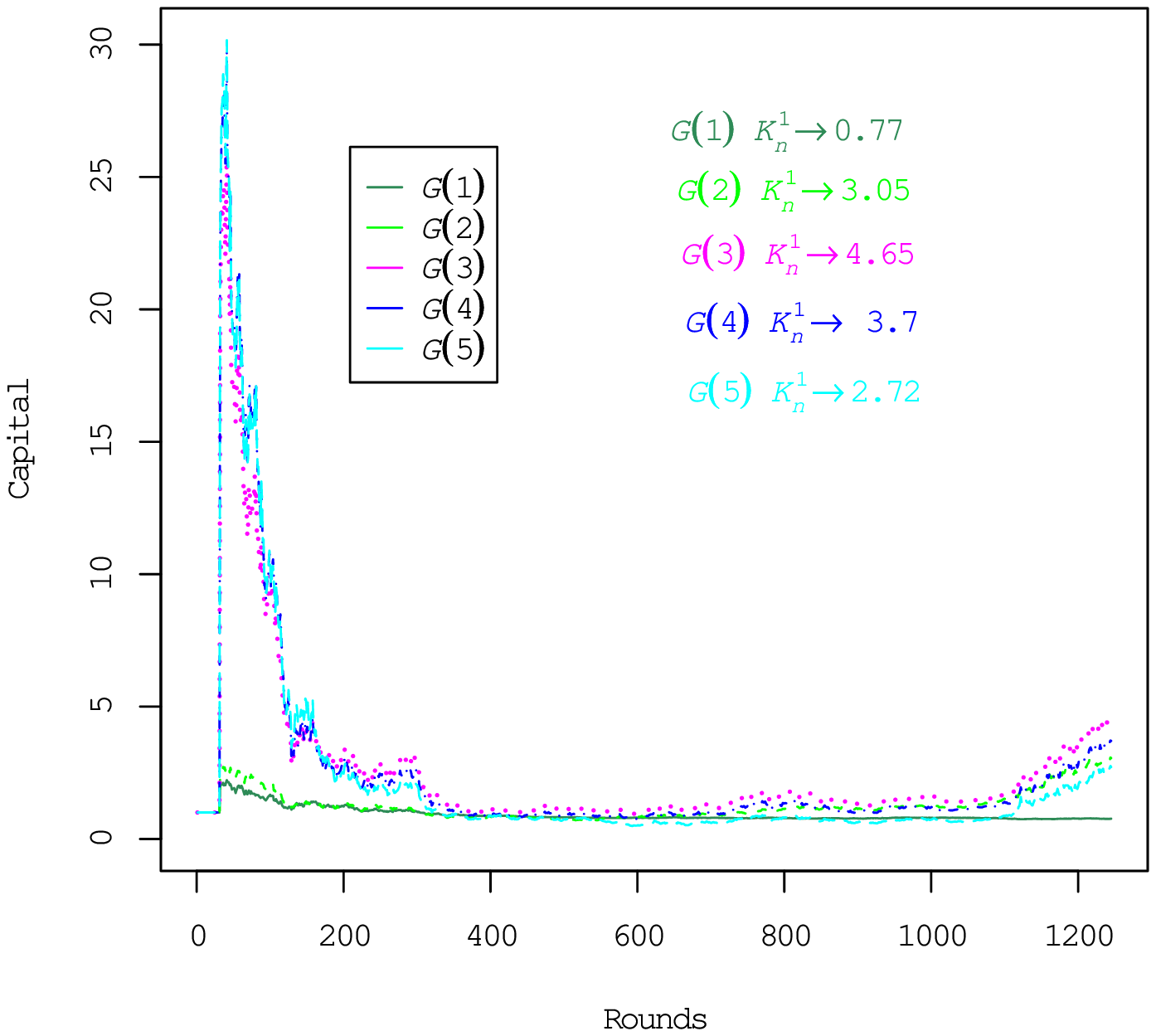}
\vspace*{-11mm}   
\caption{Capital processes $K_n^1$}
\label{fig:1-14}
\end{center}
\end{minipage}
\begin{minipage}{.5\linewidth}
\begin{center}
\includegraphics[width=8cm,height=7cm]{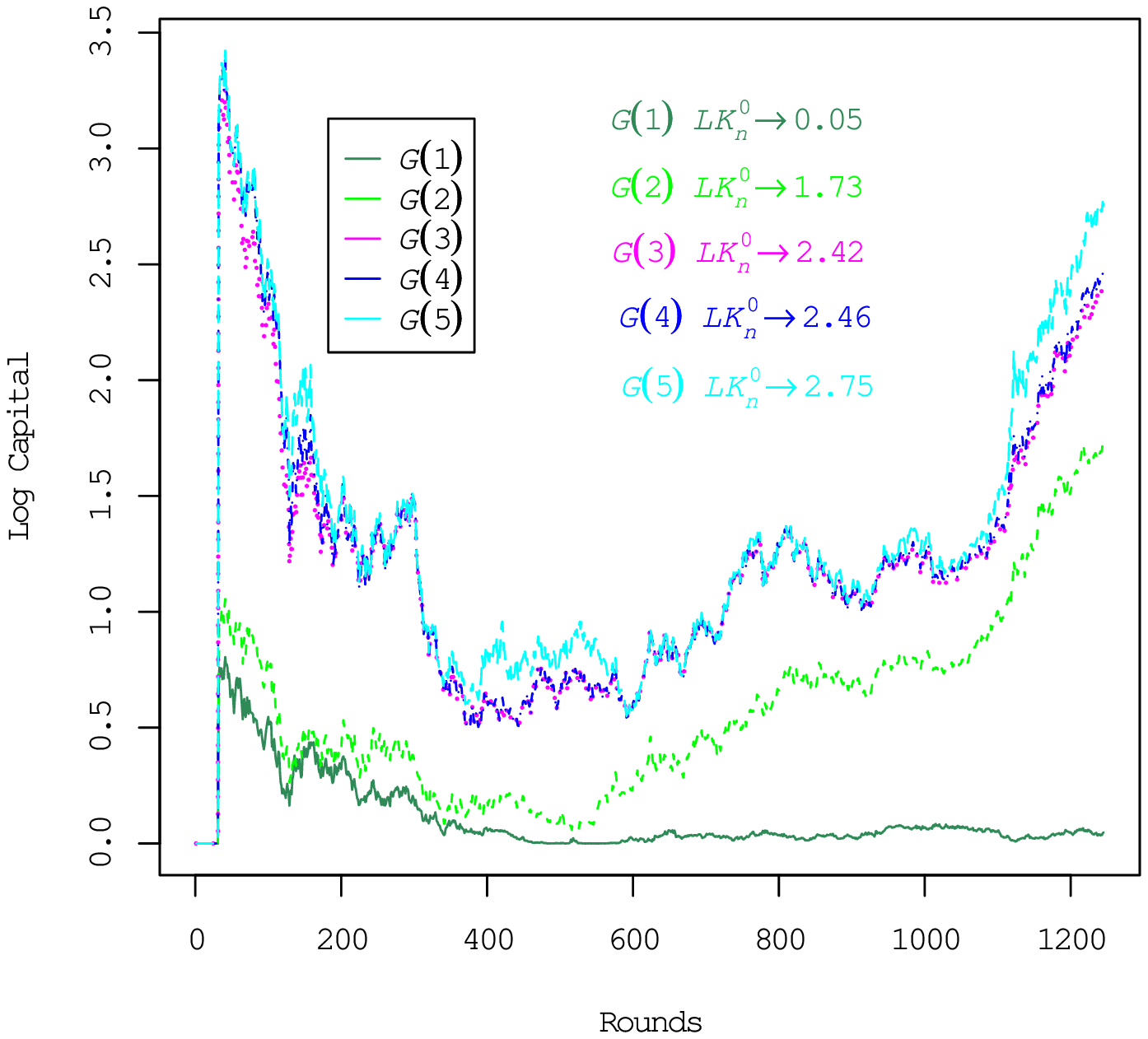}
\vspace*{-11mm}   
\caption{Log capital processes $LK_n^0$}
\label{fig:1-15}
\end{center}
\end{minipage}
\begin{minipage}{.5\linewidth}
\begin{center}
\includegraphics[width=8cm,height=7cm]{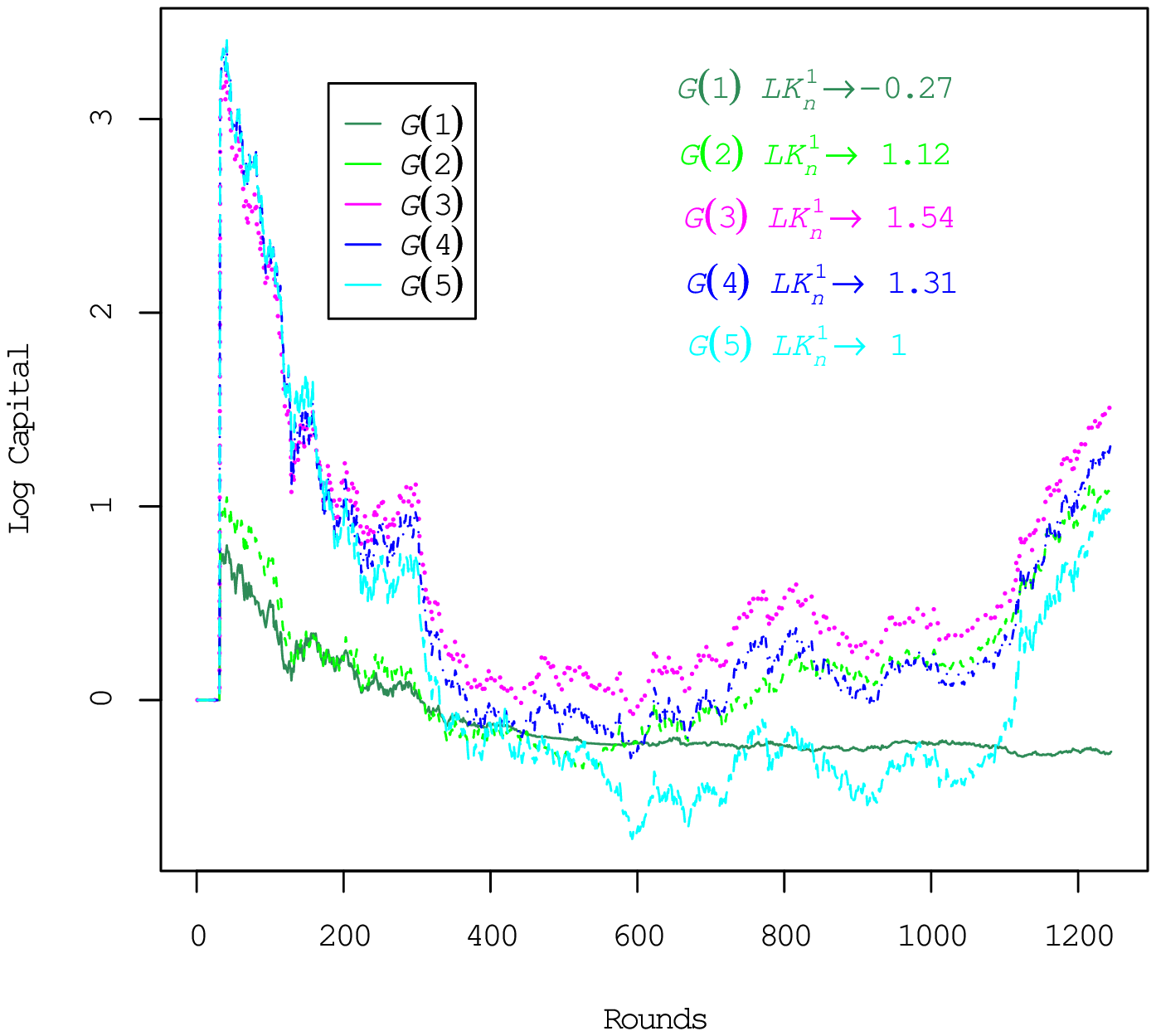}
\vspace*{-11mm}   
\caption{Log capital processes $LK_n^1$}
\label{fig:1-16}
\end{center}
\end{minipage}
\end{figure}


\begin{figure}[htbp]
\begin{minipage}{.5\linewidth}
\begin{center}
\includegraphics[width=8cm,height=7cm]{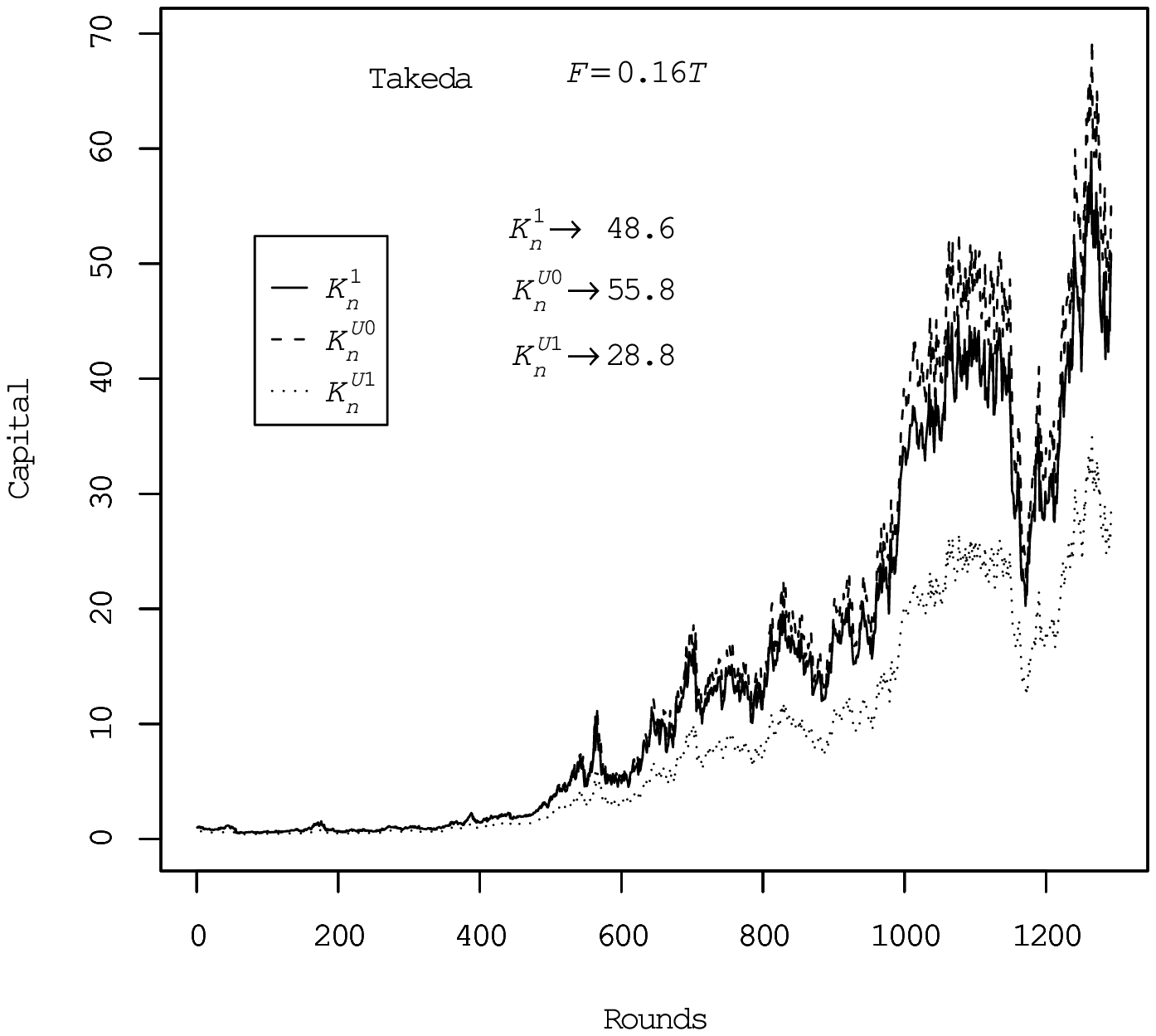}
\vspace*{-11mm}   
\caption{$K_n^1,\ K_n^{U0},\ K_n^{U1}$}
\label{fig:1-17}
\end{center}
\end{minipage}
\begin{minipage}{.5\linewidth}
\begin{center}
\includegraphics[width=8cm,height=7cm]{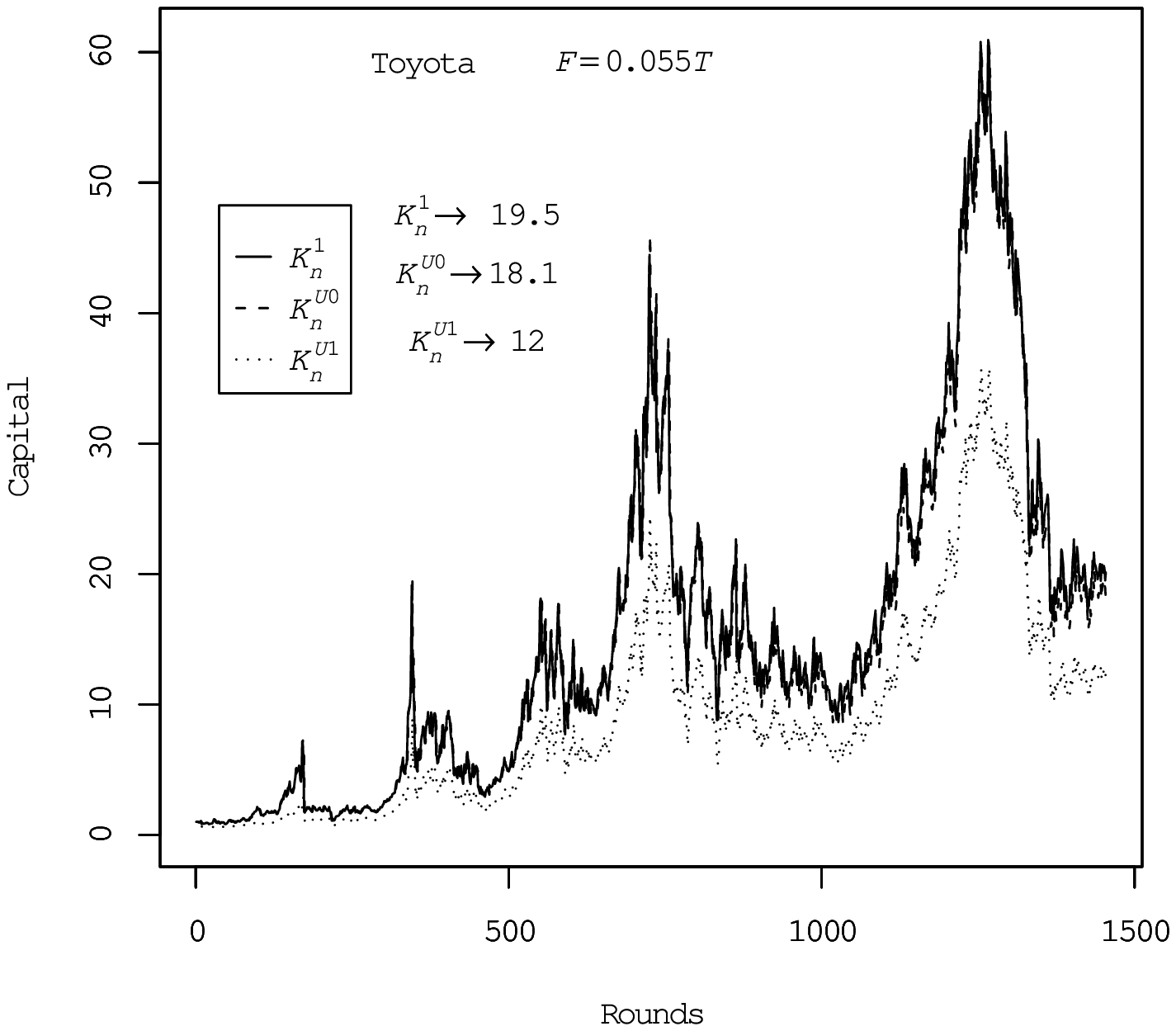}
\vspace*{-11mm}   
\caption{$K_n^1,\ K_n^{U0},\ K_n^{U1}$}
\label{fig:1-18}
\end{center}
\end{minipage}
\begin{minipage}{.5\linewidth}
\begin{center}
\includegraphics[width=8cm,height=7cm]{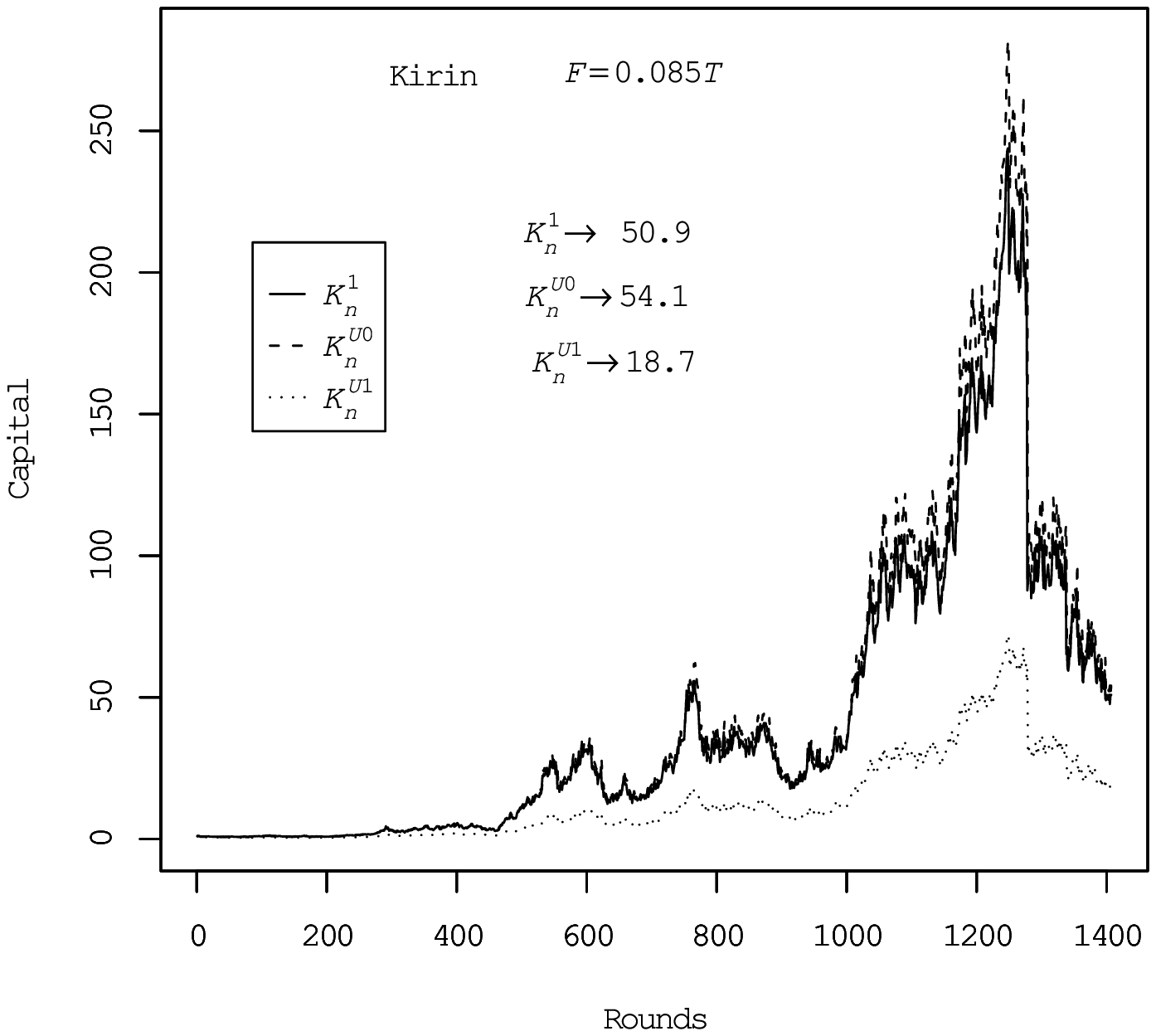}
\vspace*{-11mm}   
\caption{$K_n^1,\ K_n^{U0},\ K_n^{U1}$}
\label{fig:1-19}
\end{center}
\end{minipage}
\begin{minipage}{.5\linewidth}
\begin{center}
\includegraphics[width=8cm,height=7cm]{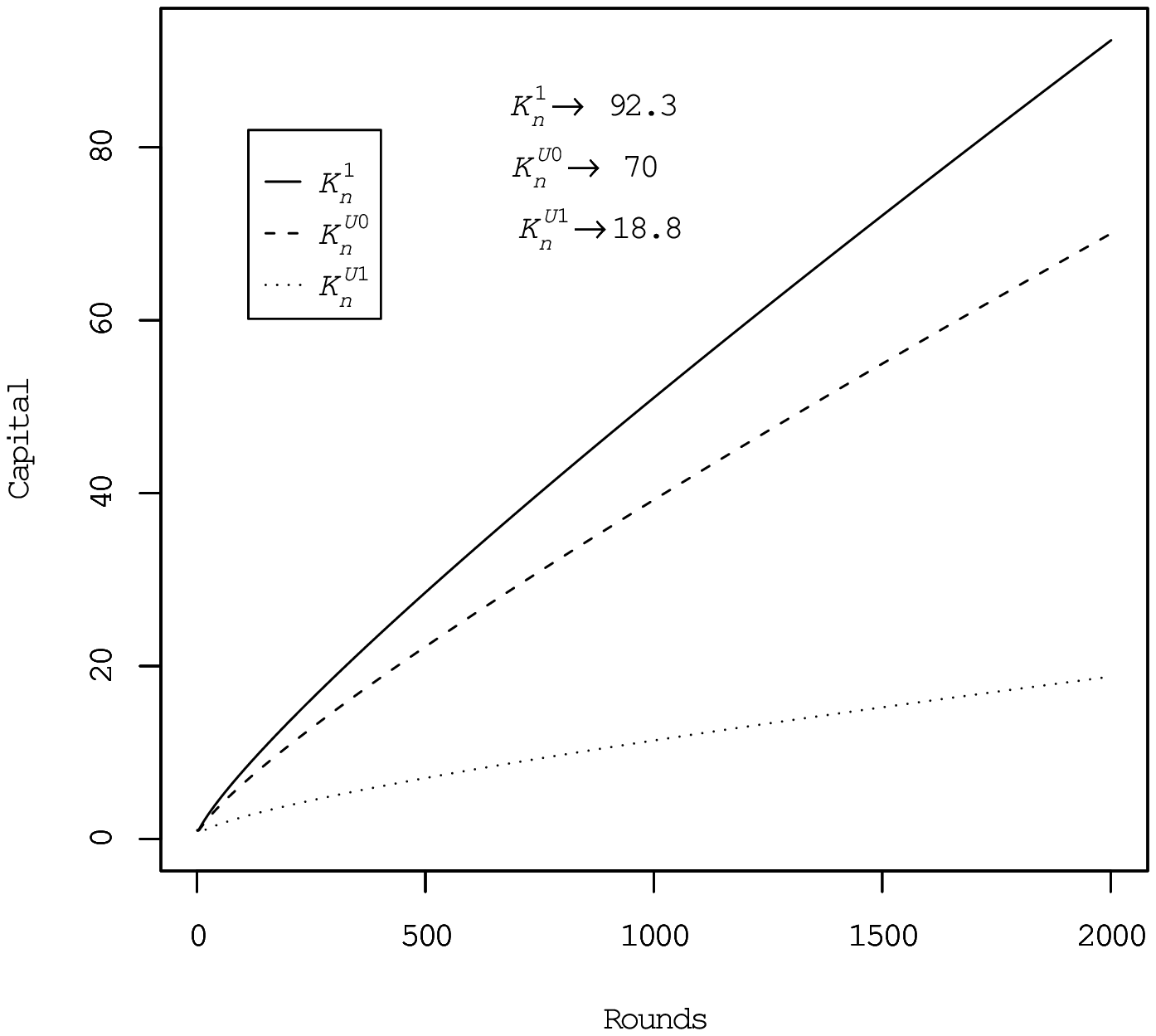}
\vspace*{-11mm}   
\caption{$K_n^1,\ K_n^{U0},\ K_n^{U1}$}
\label{fig:1-20}
\end{center}
\end{minipage}
\begin{minipage}{.5\linewidth}
\begin{center}
\includegraphics[width=8cm,height=7cm]{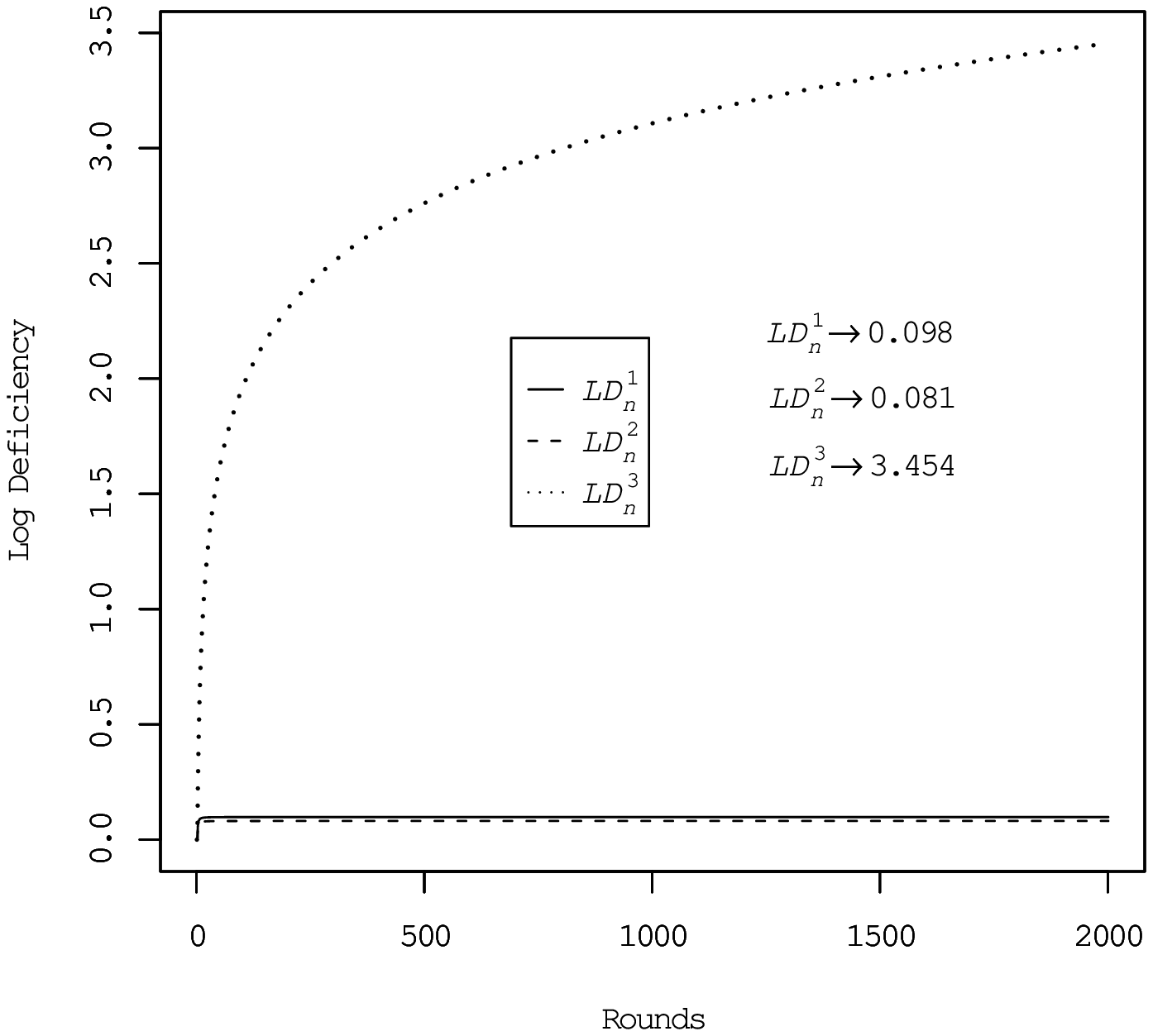}
\vspace*{-11mm}   
\caption{Log deficiency processes}
\label{fig:1-21}
\end{center}
\end{minipage}
\begin{minipage}{.5\linewidth}
\begin{center}
\includegraphics[width=8cm,height=7cm]{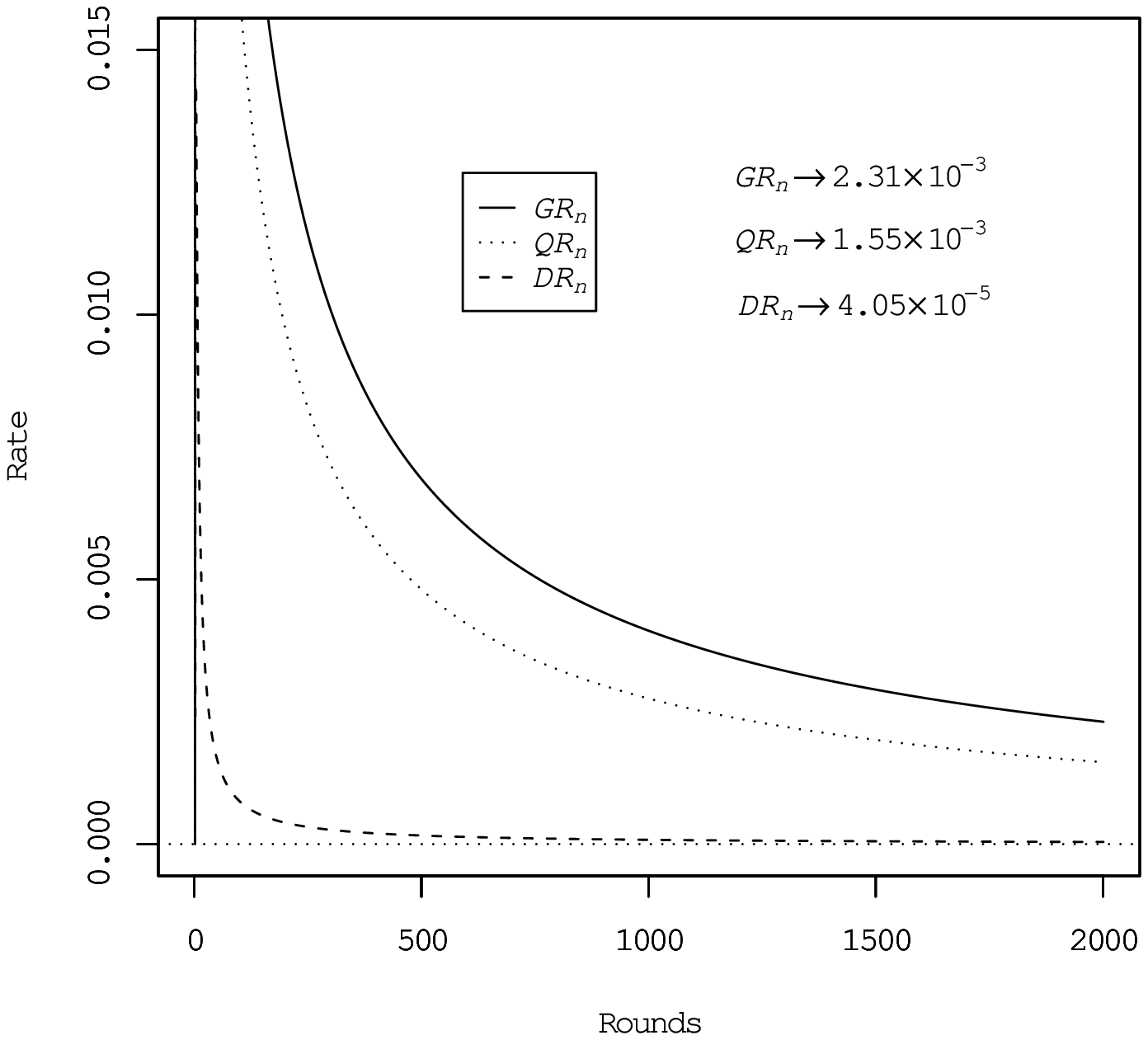}
\vspace*{-11mm}   
\caption{Rate processes}
\label{fig:1-22}
\end{center}
\end{minipage}
\end{figure}


\section{Some discussions}
\label{sec:discussion}

In this paper we proposed a sequential optimizing strategy
in multi-dimensional bounded forecasting game and showed that
it is a very flexible strategy.  From a theoretical viewpoint
it allowed us to prove 
a generalized form of the strong law of large numbers.
{}From a practical viewpoint the strategy is easy to implement even 
in high dimensions
and its performance is competitive against universal portfolio.

Theoretical comparison of our strategy with universal portfolio 
needs more detailed asymptotic investigation of 
the capital processes of these strategies.
This is left to our future research.

In Section \ref{sec:high} as a limit order type strategy we
considered successive stopping times defined by a sphere of radius
$\delta$ for the vector of returns (cf.\ (\ref{eq:1-5})), which is based on the standard Euclidean norm in ${\mathbb R}^d$. We note that other
boundaries based on other norms which are equivalent to the standard one provide the same result stated in Theorem \ref{thm:4-1}.

Theorem \ref{thm:gen-slln} 
for the case of $\sup_N V_N < \infty$ does not provide a game-theoretic
version of Kolmogorov's three series theorem.  
It only implies that $S_N$, $N=1,2,\dots$,  are bounded. However we expect
that a game-theoretic version of 
Kolmogorov's three series theorem can be established by appropriate
modification of our strategy.  This topic is also left to our future research.

\bigskip
\appendix

\section{A convergence lemma}
\label{sec:convergence-lemma}

Let $u_1, u_2, \dots$  be a sequence of points in 
${\mathbb R}^d$.  We assume  that they are bounded:
$\Vert u_n\Vert \le 1$, $\forall n$, 
and that $u_1, \dots, u_d$ are linearly independent.
Define 
\[
y_n = (u_1 u_1\tp + u_2 u_2\tp + \dots + u_{n-1} u_{n-1}\tp)^{-1} u_n \in {\mathbb R}^d.
\]
Then we have the following lemma.  It is trivial for $d=1$, but
for $d>1$ we need a careful argument.

\begin{lemma}
\label{lem:convergence}
\[
y_n \rightarrow 0, \qquad (n\rightarrow\infty).
\]
\end{lemma}

\begin{proof} We first show that $y_n$ is bounded.
Let $\lambda_{\min,d} > 0$ denote the minimum eigenvalue of 
$u_1 u_1\tp + \cdots + u_d u_d\tp$. Then
all the eigenvalues of $u_1 u_1\tp + \cdots + u_n u_n\tp$, $n\ge d$, are
greater then equal to $\lambda_{\min,d}$.
Then all the eigenvalues of $(u_1 u_1\tp  + \cdots + u_n u_n\tp)^{-2}$ 
are less than or equal to $\lambda_{\min, d}^{-2}$. Hence
\begin{equation}
\label{eq:bounded}
\Vert y_n\Vert^2 \le \lambda_{\min,d}^{-2} \Vert u_n\Vert^2
\end{equation}
and $y_n$, $n=1,2,\dots$, are bounded.

Now we argue by contradiction. Suppose that $y_n$, $n=1,2,\dots$, do not converge to zero.
Then there exists a subsequence $n_k$, $k=1,2,\dots$ 
such that 
$y_{n_k}\rightarrow a \neq 0$, $(k\rightarrow\infty)$.
In view of (\ref{eq:bounded}), if $u_{n_k}\rightarrow 0$ then 
$y_{n_k} \rightarrow 0$, which is a contradiction.
Therefore $u_{n_k}$, $k=1,2,\dots$, do not converge to 0. Then there exists a
further subsequence $\{\tilde n_k\}\subset \{n_k\}$ such that
$u_{\tilde n_k} \rightarrow b \neq 0$. Then
$y_{\tilde n_k} \rightarrow a$,  $u_{\tilde n_k} \rightarrow b$.
Consider 
\[
 (u_1 u_1\tp + \dots + u_{\tilde n_k -1} u_{\tilde n_k-1 }\tp) y_{\tilde n_k} 
= u_{\tilde n_k}.
\]
Then 
\[
 (u_1 u_1\tp  + \dots + u_{\tilde n_k-1} u_{\tilde n_k-1}\tp) y_{\tilde n_k} \rightarrow b.
\]
Multiplying by $y_{\tilde n_k}\tp$ from the left we have
\[
y_{\tilde n_k}\tp
(u_1 u_1\tp + \dots + u_{\tilde n_k-1} u_{\tilde n_k-1}\tp) y_{\tilde n_k}
=y_{\tilde n_k}\tp u_{\tilde n_k} \rightarrow a\tp b.
\]
Now the left-hand side is written as
\[
(y_{\tilde n_k}\tp u_1)^2 + \dots + 
(y_{\tilde n_k}\tp  u_{\tilde n_k-1})^2 .
\]
Note that for sufficiently large $k,k'$, 
$ (y_{\tilde n_k}\tp  u_{\tilde n_{k'}})^2$ are all close to 
$(b\tp a)^2$.  Since we have infinitely many such terms, 
the left-hand side diverges to $\infty$ if 
$b\tp a\neq 0$.
This contradicts the fact that the right-hand side converges to a finite value.
Therefore $b\tp a=0$.  But then 
\begin{align*}
&\liminf (y_{\tilde n_k}\tp u_1)^2  +  \dots + 
(y_{\tilde n_k}\tp  u_{\tilde n_k-1})^2  \ge
(y_{\tilde n_k}\tp u_1)^2 + \dots + (y_{\tilde n_k}\tp u_d)^2 \\
& \qquad \rightarrow (a\tp u_1)^2 + \dots + (a\tp u_d)^2 > 0,
\end{align*}
which is again a contradiction.
\end{proof}

We also present the following corollary of the above lemma.

\begin{corollary}
\label{cor:symmetric-half}
With the same notation and conditions  as in Lemma \ref{lem:convergence2}
\[
\tilde y_n= (u_1 u_1\tp + u_2 u_2\tp + \dots + u_{n-1} u_{n-1}\tp)^{-1/2} u_n  \rightarrow 0,
\quad (n\rightarrow \infty).
\]
\end{corollary}

This corollary follows easily 
from the fact that $\Vert \tilde y_n\Vert^2 = u_n\tp y_n$ and
$u_n$ is bounded.

Based on the above corollary 
we give a proof of Lemma \ref{lem:convergence2}.
Before going into the proof, we summarize some
facts on matrix inequalities.
For a symmetric matrix $A$, let 
$A > 0$ mean that $A$ is positive definite.
If $A\ge B >0$, then $B^{-1} \ge A^{-1} > 0$ (Lemma 4.2 of \cite{anderson-takemura}).
Note that $A\ge B\ge 0$ does not imply $A^2 \ge B^2$ 
(e.g.\ Chapter 1 of \cite{zhan}), which complicates our proof.  

\begin{proof}[Proof of Lemma \ref{lem:convergence2}]
By the definition of $C_1$ in (\ref{eq:C1})  we have
\[
V_{0,n-1}(\alpha_{n-1}^*, \alpha_n^*) \ge  \frac{1}{C_1^2} V_{0,n-1}(0,0), 
\]
where $V_{0,n-1}(0,0)=\sum_{i=-n_0+1}^n x_i x_i\tp$ 
is positive definite because of the training data.
Write
\[
\Delta\alpha_n^* = V_{0,n-1}(\alpha_{n-1}^*, \alpha_n^*)^{-1/2} 
V_{0,n-1}(\alpha_{n-1}^*, \alpha_n^*)^{-1/2} x_n(\alpha_{n-1}^*).
\]
Then
\[
\Vert \Delta\alpha_n^*  \Vert^2 
\le  \frac{x_n(\alpha_{n}^*)\tp V_{0,n-1}(\alpha_{n-1}^*, \alpha_n^*)^{-1} x_n(\alpha_n^*)}
{\lambda_{\min,0,n-1}(\alpha_{n-1}^*,\alpha_n^*)},
\]
where $\lambda_{\min,0,n-1}(\alpha_{n-1}^*,\alpha_n^*)$ is  the
minimum eigenvalue of  $V_{0,n-1}(\alpha_{n-1}^*, \alpha_n^*)$.
Let $\lambda_{\min,0,0}$ denote the 
minimum eigenvalue of  $V_{0,0}$. Then 
$\lambda_{\min,0,n-1}(\alpha_{n-1}^*,\alpha_n^*) \ge  \lambda_{\min,0,0}/C_1^2$
for all $n\ge 1$ and
\[
\Vert \Delta\alpha_n^*  \Vert^2 
\le \frac{C_1^2}{\lambda_{\min,0,0}}
 x_n(\alpha_{n}^*)\tp V_{0,n-1}(\alpha_{n-1}^*, \alpha_n^*)^{-1} x_n(\alpha_n^*).
\]
For $n\ge 1$, $1+\alpha_n^* \cdot x_n \ge \epsilon_0$. Hence 
\[
\Vert \Delta\alpha_n^*  \Vert^2 
\le \frac{C_1^2}{\epsilon_0^2 \lambda_{\min,0,0}}
 x_n\tp V_{0,n-1}(\alpha_{n-1}^*, \alpha_n^*)^{-1} x_n
\le \frac{C_1^4}{\epsilon_0^2 \lambda_{\min,0,0}}
x_n\tp V_{0,n-1}(0,0)^{-1} x_n.
\]
The right-hand side converges to 0 by Corollary 
\ref{cor:symmetric-half}.
\end{proof}

\end{document}